\documentclass[journal]{IEEEtran}
\usepackage{amsmath,amsfonts,bm}
\ifCLASSINFOpdf
\else
   \usepackage[dvips]{graphicx}
\fi
\usepackage{url}
\usepackage{cite}
\usepackage{hyperref}
\usepackage{mathtools}
\usepackage{amssymb}
\usepackage{amsthm}
\usepackage{mathrsfs}
\usepackage{algorithm}
\usepackage{color}
\usepackage{algorithmic}
\usepackage{bbding}
\newtheorem{definition}{\bf Definition}
\newtheorem{assumption}{\bf Assumption}
\newtheorem{theorem}{\bf Theorem}
\newtheorem{lemma}{\bf Lemma}

\usepackage{graphicx}

\begin{document}

\title{Riemannian Momentum Tracking: Distributed Optimization with Momentum on Compact Submanifolds}

\author{Jun Chen$^*$, Tianyi Zhu$^*$, Haishan Ye, Lina Liu, Guang Dai, Yong Liu,~\IEEEmembership{Member,~IEEE}, Yunliang Jiang, \\ Ivor W.~Tsang,~\IEEEmembership{Fellow,~IEEE}
\thanks{Jun Chen is with the National Special Education Resource Center for Children with Autism, Zhejiang Normal University, Hangzhou 311231, China, and with the School of Computer Science and Technology, Zhejiang Normal University, Jinhua 321004, China (E-mail: junc.change@gmail.com).}
\thanks{Tianyi Zhu and Lina Liu are with China Mobile Research Institute, Beijing 100032, China. (E-mail: liulina0601@gmail.com; zhu-ty@outlook.com).}
\thanks{Haishan Ye is with the School of Management, Xi'an Jiaotong University, China (E-mail: yehaishan@xjtu.edu.cn).}
\thanks{Guang Dai is with SGIT AI Lab, State Grid Corporation of China, China.}
\thanks{Yong Liu is with the Institute of Cyber-Systems and Control, Zhejiang University, Hangzhou 310027, China (E-mail: yongliu@iipc.zju.edu.cn).}
\thanks{Yunliang Jiang is with the Zhejiang Key Laboratory of Intelligent Education Technology and Application, Zhejiang Normal University, Jinhua 321004, China, and with the School of Computer Science and Technology, Zhejiang Normal University, Jinhua 321004, China , and also with the School of Information Engineering, Huzhou University, Huzhou 313000, China (E-mail: jyl2022@zjnu.cn).}
\thanks{Ivor W.~Tsang is with the Centre for Frontier Artificial Intelligence Research, A*STAR, Singapore (E-mail: ivor.tsang@gmail.com).}
\thanks{$^{*}$Equal contribution.}
}

\markboth{IEEE TRANSACTIONS ON CONTROL OF NETWORK SYSTEMS}
{Shell \MakeLowercase{\textit{et al.}}: Bare Demo of IEEEtran.cls for IEEE Journals}
\maketitle

\begin{abstract}
Gradient descent with momentum has been widely applied in various signal processing and machine learning tasks, demonstrating a notable empirical advantage over standard gradient descent. However, momentum-based distributed Riemannian algorithms have been only scarcely explored. In this paper, we propose Riemannian Momentum Tracking (RMTracking), a decentralized optimization algorithm with momentum over a compact submanifold. Given the non-convex nature of compact submanifolds, the objective function, composed of a finite sum of smooth (possibly non-convex) local functions, is minimized across agents in an undirected and connected network graph.
With a constant step-size, we establish an $\mathcal{O}(\frac{1-\beta}{K})$ convergence rate of the Riemannian gradient average for any momentum weight $\beta \in [0,1)$. Especially, RMTracking can achieve a convergence rate of $\mathcal{O}(\frac{1-\beta}{K})$ to a stationary point when the step-size is sufficiently small. To best of our knowledge, RMTracking is the first decentralized algorithm to achieve exact convergence that is $\frac{1}{1-\beta}$ times faster than other related algorithms. Finally, we verify these theoretical claims through numerical experiments on eigenvalue problems.
\end{abstract}

\begin{IEEEkeywords}
Distributed optimization, Compact submanifold, Gradient tracking, Momentum term
\end{IEEEkeywords}

\IEEEpeerreviewmaketitle

\section{Introduction}

\IEEEPARstart{I}{n} large-scale systems such as signal processing, control, and machine learning, data is often distributed across numerous nodes, making it challenging for a centralized server to manage the increasing computational demands. As a result, 
distributed optimization has become increasingly important in recent years due to the challenges posed by large-scale datasets and complex multi-agent systems. Among the explored approaches, the distributed sub-gradient method is notable for its simplicity in combining local gradient descent to reduce the consensus error~\cite{nedic2009distributed,yuan2016convergence}. To achieve exact convergence to an $\epsilon$-stationary point, various algorithms have leveraged both local and global historical information. For instance, the gradient tracking algorithm~\cite{qu2017harnessing,yuan2018exact}, primal-dual framework~\cite{alghunaim2020decentralized}, and ADMM~\cite{shi2014linear,aybat2017distributed} have been explored with convex local functions.

Let $\mathcal{M}$ represent a compact submanifold of $\mathbb{R}^{d\times r}$, such as the Stiefel manifold or the Grassmann manifold \cite{deng2023decentralized}. We consider the following distributed non-convex (but smooth) optimization problem over a compact submanifold: 

\begin{equation}
\begin{aligned}
&\min \frac{1}{n} \sum_{i=1}^n f_i\left(x_i\right), \\
&\text { s.t. }  x_1=x_2=\cdots=x_n, \quad x_i \in \mathcal{M}, \quad \forall i \in [n],
\end{aligned}
\label{decentralized}
\end{equation}
where $n$ represents the number of agents, and $f_i$ denotes the local function for each agent. Many important large-scale tasks can be formulated as this distributed optimization problem, such as the principle component analysis~\cite{ye2021deepca}, eigenvalue estimation~\cite{chen2021decentralized}, dictionary learning~\cite{raja2015cloud}, and deep neural networks with orthogonal constraint~\cite{vorontsov2017orthogonality,huang2018orthogonal,eryilmaz2022understanding}.
 
\begin{table*}[t]
	\caption{Comparison with existing algorithms based on the criterion $\|\operatorname{grad} f\|^2 \leq \epsilon$. ``GT" represents the gradient tracking.}
	\begin{center}
		\begin{tabular}{cccccccccc}
			\hline
			  Methods & GT & $\alpha$ & Order & Manifold & Operator & $\Vert \hat{g}_k \Vert^2$ & $\frac{1}{n}\Vert \mathbf{x}_{k} - \bar{\mathbf{x}}_{k} \Vert^2$ & Convergence rate  & Momentum \\
			\hline
            DRDGD~\cite{chen2021decentralized}  & \XSolidBrush & $\mathcal{O}(\frac{1}{\sqrt{k}})$ & First-order & Stiefel & Retraction & / & / & $\mathcal{O}\left(\frac{1}{\sqrt{K}}\right)$ & \XSolidBrush \\
            DPRGD~\cite{deng2023decentralized} & \XSolidBrush & $\mathcal{O}(\frac{1}{\sqrt{k}})$ & First-order & Compact & Projection & / & / & $\mathcal{O}\left(\frac{1}{\sqrt{K}}\right)$ & \XSolidBrush \\
            DRCGD~\cite{chen2024decentralized} & \XSolidBrush & $\mathcal{O}(\frac{1}{\sqrt{k}})$ & First-order & Stiefel & Projection & / & / & / & \XSolidBrush \\
            \hline
            DRGTA~\cite{chen2021decentralized} & \CheckmarkBold & $\mathcal{O}(1)$ & First-order & Stiefel & Retraction & $\mathcal{O}\left(\frac{1}{\alpha K}\right)$ & $\mathcal{O}\left(\frac{1}{K}\right)$ & $\mathcal{O}\left(\frac{1}{\alpha K}\right)$ & \XSolidBrush \\
            DPRGT~\cite{deng2023decentralized} & \CheckmarkBold & $\mathcal{O}(1)$ & First-order & Compact & Projection & $\mathcal{O}\left(\frac{1}{\alpha K}\right)$ & $\mathcal{O}\left(\frac{1}{K}\right)$ & $\mathcal{O}\left(\frac{1}{\alpha K}\right)$ & \XSolidBrush \\
            DRNGD~\cite{hu2023decentralized} & \CheckmarkBold & $\mathcal{O}(1)$ & Second-order & Compact & Retraction & / & $\mathcal{O}\left(\frac{\alpha}{K}\right)$ & $\mathcal{O}\left(\frac{1}{\alpha K}\right)$ & \XSolidBrush \\
            \hline
            \textbf{This paper} & \CheckmarkBold & $\mathcal{O}(1)$ & First-order & Compact & Projection & $\mathcal{O}\left(\frac{1-\beta}{\alpha K}\right)$ & $\mathcal{O}\left(\frac{1}{K}\right)$ & $\mathcal{O}\left(\frac{1}{\alpha K}\right)$ & \CheckmarkBold \\
            \textbf{This paper} (small $\alpha$) & \CheckmarkBold & $\mathcal{O}(1)$ & First-order & Compact & Projection & $\mathcal{O}\left(\frac{1-\beta}{\alpha K}\right)$ & $\mathcal{O}\left(\frac{1-\beta}{K}\right)$ & $\mathcal{O}\left(\frac{1-\beta}{\alpha K}\right)$ & \CheckmarkBold \\
			\hline
		\end{tabular}
	\end{center}
\label{summary}
\end{table*}

Existing Euclidean distributed composite optimization methods cannot solve the problem in (\ref{decentralized}) due to the inherent non-convexity of the compact submanifold. On the other hand, Riemannian optimization, which involves optimization on Riemannian manifolds, presents additional challenges in constructing a consensus step~\cite{absil2008optimization,boumal2019global,sato2021riemannian}. For example, a straightforward approach involves computing the arithmetic average in Euclidean space is not viable on Riemannian manifolds, as the resulting mean may fall outside the manifold. To overcome this limitation, Ref. \cite{shah2017distributed} developed a Riemannian consensus approach, but it requires an asymptotically infinite number of consensus steps to achieve convergence.
In response, Ref. \cite{wang2022decentralized} integrated an augmented Lagrangian function with the gradient tracking algorithm to achieve single-step consensus for convergence.

When $\mathcal{M}$ is the Stiefel manifold, an embedded submanifold in Euclidean space \cite{chen2023local}, Ref. \cite{chen2021decentralized} proposed a distributed Riemannian gradient descent (DRDGD) algorithm over the Stiefel manifold, achieving exact convergence within a finite number of consensus steps at a rate of $\mathcal{O}(\frac{1}{\sqrt{K}})$. Additionally, it developed the gradient tracking variant (DRGTA), which reaches an $\epsilon$-stationary point with a convergence rate of $\mathcal{O}(\frac{1}{K})$. More recently, Ref. \cite{chen2024decentralized} proposed the first distributed Riemannian conjugate gradient descent (DRCGD) algorithm and established its global convergence on the Stiefel manifold.
Building on these advancements, Ref. \cite{deng2023decentralized} extended the scope from the Stiefel manifold to the compact submanifold. Specifically, it employed projection operators instead of retractions, resulting in the distributed projected Riemannian gradient descent (DPRGD) algorithm for compact submanifolds with the convergence rate of $\mathcal{O}(\frac{1}{\sqrt{K}})$. Similarly, the gradient tracking variant (DPRGT) was developed, reaching an $\epsilon$-stationary point with a convergence rate of $\mathcal{O}(\frac{1}{K})$. Recently, Ref. \cite{zhao2024distributed} proposed a distributed Riemannian stochastic gradient tracking algorithm on the Stiefel manifold, which integrates gradient tracking with a variable sample-size strategy for gradient approximation. Since the efficiency of distributed optimization is often hindered by communication bottlenecks, Ref. \cite{chen2025decentralized} proposed the Quantized Riemannian Gradient Tracking (Q-RGT) algorithm, which is free from retraction and projection, and iterates within the neighborhood of compact submanifolds.

Inspired by Ref. \cite{takezawamomentum}, this paper introduces the momentum term into distributed Riemannian optimization. The concept of the momentum originates from Polyak's heavy-ball method \cite{polyak1964some}. Gradient descent with heavy-ball method establishes an accelerated linear convergence rate for strongly convex and smooth problems \cite{ghadimi2015global}. On Riemannian manifolds, momentum-based approaches have also been developed. Ref. \cite{becigneul2018riemannian} intrinsically generalized the momentum term to the Riemannian setting. Ref. \cite{han2020riemannian} proposed a stochastic recursive momentum algorithm that achieves the complexity of $\mathcal{O}(\epsilon^{-3})$ to find $\epsilon$-approximate solution for Riemannian optimization. Ref. \cite{alimisis2021momentum} exploited momentum to speed up convergence for geodesically convex loss functions on Riemannian manifolds. Despite the extensive study of momentum-based accelerated variants in Euclidean space and on Riemannian manifolds, there has been limited exploration of such techniques in distributed Riemannian scenarios.

\subsection{Contribution}

This paper presents Riemannian momentum tracking algorithm to address problem (\ref{decentralized}) under a connected graph. Table \ref{summary} provides a summary of existing approaches to solve this problem. Our contributions can be summarized as follows: 
\begin{enumerate}
    \item We propose a novel Riemannian Momentum Tracking (RMTracking) method over a compact submanifold to solve the problem in (\ref{decentralized}). To support this, we develop the Lipschitz-type inequalities for the momentum term on any compact $C^2$-submanifold in Euclidean space.
    \item With a constant step-size $\alpha=\mathcal{O}(1)$, we establish an $\mathcal{O}(\frac{1-\beta}{K})$ convergence rate of the Riemannian gradient average for any momentum weight $\beta \in [0,1)$. When $\alpha$ is sufficiently small, RMTracking can achieve a convergence rate of $\mathcal{O}(\frac{1-\beta}{K})$ to a stationary point.
    \item We validate the effectiveness of RMTracking through numerical experiments on eigenvalue problems, demonstrating its superiority over state-of-the-art algorithms.
\end{enumerate}

\subsection{Notation}

The undirected and connected graph $G=(\mathcal{V},\mathcal{E})$ is defined by the set of agents $\mathcal{V}=\{1,2,\cdots,n\}$ and the set of edges $\mathcal{E}$. The adjacency matrix $W$ of this topological graph satisfies $W_{ij}>0$ and $W_{ij}=W_{ji}$ if there exists an edge $(i,j) \in \mathcal{E}$; otherwise, $W_{ij}=0$. We represent the collection of all local variables $x_i$ as $\mathbf{x}$, stacking them as $\mathbf{x}^\top:=(x_1^\top, \cdots, x_n^\top)$. The function $f(\mathbf{x})$ is defined as the average of local functions: $\frac{1}{n} \sum_{i=1}^n f_i(x_i)$. The $m$-fold Cartesian product of $\mathcal{M}$ is denoted as $\mathcal{M}^n=\mathcal{M} \times \mathcal{M} \times \cdots \times\mathcal{M}$. We use $\emph{I}_d$ to represent the $d \times d$ identity matrix and $\textbf{1}_n \in \mathbb{R}^n$ as a vector with all entries equal to one. Furthermore, we define $\mathbf{W}^t:=W^t \otimes \emph{I}_d$, where $\otimes$ denotes the Kronecker product and $t$ is a positive integer.

\section{Preliminaries}
\subsection{Compact Submanifold}
We define the distance between $x \in \mathbb{R}^{d\times r}$ and $\mathcal{M}$ by
\begin{equation}
\operatorname{dist}(x,\mathcal{M}):=\inf_{y \in \mathcal{M}} \Vert y-x\Vert.
\end{equation}
For any radius $\Gamma>0$, the $\Gamma$-tube around $\mathcal{M}$ can be defined as the neighborhood set:
\begin{equation}
U_{\mathcal{M}}(\Gamma):=\{x: \operatorname{dist}(x,\mathcal{M}) \leq \Gamma \}.
\end{equation}
Furthermore, we can define the nearest-point projection of $x \in \mathbb{R}^{d\times r}$ onto $\mathcal{M}$ by
\begin{equation}
\mathcal{P}_{\mathcal{M}}(x):= \arg \min_{y \in \mathcal{M}} \Vert y - x \Vert.
\end{equation}

\begin{definition}
\label{smooth}
\cite{clarke1995proximal} The \textbf{$\Gamma$-proximally smooth} set on $\mathcal{M}$ satisfies, for $\forall x, y \in U_{\mathcal{M}}(\gamma)$, that \\
(i) For any real $\gamma \in (0,\Gamma)$, the estimate holds:
\begin{equation}
\left\|\mathcal{P}_{\mathcal{M}}(x)-\mathcal{P}_{\mathcal{M}}(y)\right\| \leq \frac{\Gamma}{\Gamma-\gamma}\|x-y\|;
\label{lip_ineq}
\end{equation}
(ii) For any point $x \in \mathcal{M}$ and a normal $v \in N_x \mathcal{M}$, the following inequality holds for all $y \in \mathcal{M}$:
\begin{equation}
\langle v, y-x\rangle \leq \frac{\|v\|}{2 \Gamma}\|y-x\|^2.
\end{equation}
\end{definition}

Based on Definition~\ref{smooth}, it can be said that if a closed set $\mathcal{M}$ is $\Gamma$-proximally smooth, then the projection $\mathcal{P}_{\mathcal{M}}(x)$ is uniquely defined within a neighborhood $\operatorname{dist}(x,\mathcal{M}) < \Gamma$. The Stiefel manifold exhibits $1$-proximally smooth, whereas the Grassmann manifold is $1/\sqrt{2}$-proximally smooth~\cite{balashov2021gradient}. Furthermore, Eq.(\ref{lip_ineq}) can be simplified to the following property
\begin{equation}
\left\|\mathcal{P}_{\mathcal{M}}(x)-\mathcal{P}_{\mathcal{M}}(y)\right\| \leq R \|x-y\|,
\label{property}
\end{equation}
where $R>1$. The inequality will be used to characterize the local convergence of distributed Riemannian gradient methods and associated neighborhood.

\subsection{Consensus Problem}

Let $x_i \in \mathcal{M}$ represent the local variables of each agent $i \in [n]$. Then, the Euclidean average point of $x_1, x_2, \cdots, x_n$ is denoted by
\begin{equation}
    \hat{x}:=\frac{1}{n} \sum_{i=1}^n x_i .
\end{equation}
In Euclidean space, the consensus error can be measured by $\Vert x_i - \hat{x}\Vert$. Over a compact embedded submanifold, it follows from the induced arithmetic mean~\cite{sarlette2009consensus} that:
\begin{equation}
    \bar{x}\in \mathcal{P}_{\mathcal{M}}(\hat{x}).
\end{equation}
The Riemannian gradient of $f_i(x)$, utilizing the induced Riemannian metric induced by the Euclidean inner product $\langle \cdot,\cdot \rangle$, can be expressed as
\begin{equation}
\operatorname{grad} f_i(x)=\mathcal{P}_{T_x \mathcal{M}}(\nabla f_i(x)),
\label{riemann}
\end{equation}
where $\mathcal{P}_{T_x \mathcal{M}}(\cdot)$ is the orthogonal projection onto $T_x \mathcal{M}$~\cite{edelman1998geometry,absil2008optimization}. Subsequently, we define the $\epsilon$-stationary point of distributed problem based on the following definition.

\begin{definition}
\cite{chen2021decentralized} The set of points $\mathbf{x}^\top=(x_1^\top, \cdots, x_n^\top)$ is termed an \textbf{$\epsilon$-stationary} point of distributed problem~(\ref{decentralized}) if the following condition holds:
\begin{equation}
\frac{1}{n} \sum_{i=1}^n\left\|x_i-\bar{x}\right\|^2 \leq \epsilon,  \quad \|\operatorname{grad} f(\bar{x})\|^2 \leq \epsilon ,
\end{equation}
where we have $f(\bar{x})=\frac{1}{n} \sum_{i=1}^n f_i(\bar{x})$.
\label{stationary}
\end{definition}

To reach the stationary point, it is important to address the consensus problem over $\mathcal{M}$, which involves minimizing the following quadratic loss function:
\begin{equation}
\begin{aligned}
&\min \varphi^t(\mathbf{x}):=\frac{1}{4} \sum_{i=1}^n \sum_{j=1}^n W_{i j}^t\left\|x_i-x_j\right\|^2, \\
&\text { s.t. } \quad x_i \in \mathcal{M}, \quad \forall i \in[n] \text {, } \\
\end{aligned}
\label{consensus}
\end{equation}
where the positive integer $t$ represents the $t$-th power of $W$. It is important to highlight that the doubly stochastic matrix $W_{ij}^t$ is calculated by performing $t$ communication steps on $T_x \mathcal{M}$ and satisfies the following assumption.

\begin{assumption}
We assume that for the undirected and connected graph $G$, the doubly stochastic matrix $W$ satisfies: (a)
$W = W^\top$; (b) $0 < W_{ii} < 1$ and $W_{ij} \geq 0$; (c) the eigenvalues of $W$ lie in $(-1, 1]$. Additionally, the second largest
singular value $\sigma_2$ of $W$ falls within $[0, 1)$.
\label{weight}
\end{assumption}

\subsection{Lipschitz Condition}

Throughout the paper, we assume that the local objective function $f_i(x)$ is Lipschitz smooth, which is a standard assumption in the optimization problem~\cite{jorge2006numerical,zeng2018nonconvex,deng2023decentralized}.
\begin{assumption}
For each agent $i$, the local objective function $f_i(x)$ has $L$-Lipschitz continuous gradient
\begin{equation}
\Vert \nabla f_i(x) - \nabla f_i(y) \Vert \leq L \Vert x - y \Vert , \quad i \in [n],
\end{equation}
and let $L_f:= \max_{x \in \mathcal{M}}\Vert \nabla f_i(x) \Vert$. Therefore, $\nabla f(x)$ is also $L$-Lipschitz continuous in Euclidean space and $L_f \geq \max_{x \in \mathcal{M}}\Vert \nabla f(x) \Vert$.
\label{lipschitz}
\end{assumption}

Under Assumption~\ref{lipschitz}, we present a quadratic upper bound of $f_i$:
\begin{equation}
    f_i(y) \leq f_i(x) + \langle \nabla f_i(x), y-x \rangle + \frac{L}{2} \Vert y-x \Vert^2 , \quad i \in [n].
\label{lipschitz2}
\end{equation}
With the properties of projection operators, we can derive two similar Lipschitz inequalities on a compact submanifold as the Euclidean-type ones~\cite{nesterov2013introductory} in the following lemma.

\begin{lemma}
Under Assumption~\ref{lipschitz}, for any $x,y \in \mathcal{M}$, if $f(x)$ is $L$-Lipschitz smooth in Euclidean space, there exists a constant $L_g=L+\frac{1}{\Gamma}L_f$ such that

\begin{equation}
    f_i(y) \leq f_i(x) + \langle \operatorname{grad} f_i(x), y-x \rangle + \frac{L_g}{2} \Vert y-x \Vert^2 ,
\end{equation}

\begin{equation}
    \Vert \operatorname{grad} f_i(x) - \operatorname{grad} f_i(y) \Vert \leq L_g \Vert x-y \Vert , \quad i \in [n].
\end{equation}

\label{lem1}
\end{lemma}
\begin{proof}
The proofs can be found in \cite{deng2023decentralized}.
\end{proof}

\section{Distributed Momentum Algorithms over a Compact Submanifold}

In this section, we present two distributed momentum algorithms over a compact submanifold, Distributed Riemannian Gradient Tracking with Momentum (DRGTM) and Riemannian Momentum Tracking (RMTracking) for solving the problem~(\ref{decentralized}).

By incorporating gradient tracking~\cite{qu2019accelerated,chen2021decentralized,deng2023decentralized}, the distributed algorithm introduces an auxiliary variable $s_{i,k}$ for each agent $i$ at each iteration. This variable is designed to track the average of the gradients $\nabla f_i(x_{i,k})$ for all agents $i=1,\cdots,n$. If $x_{i,k}$ converges to some point $x_{\infty}$, then $s_{i,k}$ converges to $\frac{1}{n} \sum_{i=1}^n \nabla f_i(x_{\infty})$. Specifically, each agent $i \in [n]$ utilizes this auxiliary variable $s_{i,k}$ to asymptotically track the average descent direction in distributed scenarios, which is updated recursively as follows:
\begin{equation}
    s_{i,k+1}=\sum_{j=1}^n W_{i j}^t s_{j,k} + \nabla f_i(x_{i,k+1}) - \nabla f_i(x_{i,k}) .
\end{equation}

A crucial advantage of gradient tracking-type methods lies in the applicability of a constant step-size $\alpha$. Given the communication network with an adjacency matrix $W$, each agent $i \in [n]$ can update $x_{i,k}$ in Euclidean space by:
\begin{equation}
    x_{i,k+1}=\sum_{j=1}^n W_{i j}^t x_{j,k}-\alpha s_{i,k},
\end{equation}
where $\alpha > 0$ is the step-size. Similar to the setting~\cite{deng2023decentralized}, we use the projection operator to compute the Riemannian gradient update, for all $i \in [n]$,
\begin{equation}
\begin{aligned}
    x_{i,k+1}& =\mathcal{P}_{\mathcal{M}}\left(\sum_{j=1}^n W_{i j}^t x_{j,k}-\alpha \mathcal{P}_{T_{x_{i,k}}\mathcal{M}}(s_{i,k})\right), \\
    s_{i,k+1}&=\sum_{j=1}^n W_{i j}^t s_{j,k} + \operatorname{grad} f_i(x_{i,k+1}) - \operatorname{grad} f_i(x_{i,k}) .
\end{aligned}
\end{equation}
In this case, Riemannian gradient is computed through Eq.(\ref{riemann}), and the projection of the variable $s_{i,k}$ to the tangent space $T_{x_{i,k}} \mathcal{M}$ is used in updating $x_{i,k+1}$.

\subsection{Algorithm DRGTM}

In distributed optimization, a straightforward approach using momentum is distributed SGD with momentum (DSGDM) \cite{takezawa2023momentum}. In Euclidean space, DSGDM updates two variables $m_{i,k}$, $x_{i,k}$ as follows:
\begin{equation}
\begin{aligned}
    m_{i,k+1} &= \beta m_{i,k} + \nabla f_i(x_{i,k}), \\
    x_{i,k+1}&=\sum_{j=1}^n W_{i j}^t x_{j,k}-\alpha m_{i,k+1}, 
\end{aligned}
\end{equation}
where $m_{i,k}$ is the local momentum of node $i$ and $\beta \in [0,1)$ is a  momentum weight. By applying the gradient tracking, we further update the rule as follows:
\begin{equation}
\begin{aligned}
    m_{i,k+1} &= \beta m_{i,k} + s_{i,k}, \\
    x_{i,k+1}&=\sum_{j=1}^n W_{i j}^t x_{j,k}-\alpha m_{i,k+1},  \\
    s_{i,k+1} &=\sum_{j=1}^n W_{i j}^t s_{j,k} + \nabla f_i(x_{i,k+1}) - \nabla f_i(x_{i,k}) .
\end{aligned}
\end{equation}

To accelerate the gradient tracking on a compact submanifold, the distributed Riemannian gradient tracking with momentum can be, for all $i \in [n]$,
\begin{equation}
\begin{aligned}
    m_{i,k+1}& = \beta m_{i,k} + \mathcal{P}_{T_{x_{i,k}}\mathcal{M}}(s_{i,k}), \\
    x_{i,k+1}& =\mathcal{P}_{\mathcal{M}}\left(\sum_{j=1}^n W_{i j}^t x_{j,k}-\alpha m_{i,k+1} \right), \\
    s_{i,k+1}&=\sum_{j=1}^n W_{i j}^t s_{j,k} + \operatorname{grad} f_i(x_{i,k+1}) - \operatorname{grad} f_i(x_{i,k}) .
\end{aligned}
\label{momentum1}
\end{equation}

\subsection{Algorithm RMTracking}

\begin{algorithm}[htbp]
	\caption{Riemannian Momentum Tracking (RMTracking)}
	\label{alg2}
	\begin{algorithmic}[1]
		\REQUIRE
		Initial point $\mathbf{x}_0 \in \mathcal{M}^n$, an integer $t$, the step-size $\alpha$, the momentum weight $\beta$. Let $s_{i,0}=\operatorname{grad} f_i(x_{i,0})$. \\
            \algorithmiccomment{for each node $i \in [n]$, in parallel}
		\FOR{$k=0,\cdots$} 
		\STATE Project onto tangent space $v_{i,k}=\mathcal{P}_{T_{x_{i,k}}\mathcal{M}}(s_{i,k})$ 
        \STATE Update $x_{i,k+1}=\mathcal{P}_{\mathcal{M}}\left(\sum_{j=1}^n W_{i j}^t x_{j,k}-\alpha v_{i,k}\right)$
        \STATE Momentum update $m_{i,k+1} = \beta m_{i,k} + \operatorname{grad} f_i(x_{i,k})$
        \STATE Riemannian momentum tracking \\ $s_{i,k+1}=\sum_{j=1}^n W_{i j}^t s_{j,k} + m_{i,k+1} - m_{i,k}$
		\ENDFOR
	\end{algorithmic}
\end{algorithm}

However, the DRGTM algorithm applies the momentum term after Riemannian gradient tracking, which can cause the momentum-incurred bias in updating the variable $x_{i,k+1}$ \cite{yuan2021decentlam}. Inspired by Momentum Tracking \cite{takezawamomentum}, we suggest swapping the order of Riemannian gradient tracking and momentum term to mitigate the inconsistency bias through gradient tracking-like techniques. Specifically, we introduce Riemannian Momentum Tracking, which utilizes an auxiliary variable $m_{i,k}$ to track the average momentum parameter based on Riemannian gradient $\operatorname{grad} f_i(x_{i,k})$, for all $i \in [n]$,
\begin{equation}
\begin{aligned}
    x_{i,k+1}& =\mathcal{P}_{\mathcal{M}}\left(\sum_{j=1}^n W_{i j}^t x_{j,k}-\alpha \mathcal{P}_{T_{x_{i,k}}\mathcal{M}}(s_{i,k}) \right), \\
    m_{i,k+1}& = \beta m_{i,k} + \operatorname{grad} f_i(x_{i,k}), \\
    s_{i,k+1}&=\sum_{j=1}^n W_{i j}^t s_{j,k} + m_{i,k+1} -  m_{i,k} .
\end{aligned}
\label{momentum}
\end{equation}
Although this adjustment merely reorders the application of the momentum term, it is a critical step in improving the convergence rate compared to DRGTM. Therefore, the following theoretical analysis will focus on RMTracking. The detailed description is presented in Algorithm~\ref{alg2}.

\section{Convergence Analysis of RMTracking}

\subsection{Technical Lemmas}

According to~\cite{chen2024decentralized}, the analysis of RMTracking is restricted to the neighborhood $\mathcal{N}:=\{\mathbf{x}: \Vert \hat{x}-\bar{x} \Vert \leq \Gamma/2\}$. This means $\mathbf{x}_k \in \mathcal{N}$ when $t$ and $\alpha$ satisfy certain conditions.

For ease of notation, let us denote the definition of variables
\[
\begin{aligned}
    & g_{i,k} :=  \operatorname{grad} f_i (x_{i,k}), \;\; \hat{g}_k := \frac{1}{n} \sum_{i=1}^n \operatorname{grad} f_i (x_{i,k}), \\
    & \mathbf{v}_k :=[v_{1,k}^\top, \cdots, v_{n,k}^\top]^\top, \;\; \mathbf{s}_k :=[s_{1,k}^\top, \cdots, s_{n,k}^\top]^\top, \\
    & \mathbf{m}_k :=[m_{1,k}^\top, \cdots, m_{n,k}^\top]^\top, \;\; \hat{m}_{k} := \frac{1}{n} \sum_{i=1}^n m_{i,k} .
\end{aligned}
\]
Without loss of generality, it follows from Eq.(\ref{momentum}) that we can take $m_{i,0}=0$, and express the momentum term $m_{i,k}$ as
\begin{equation}
    m_{i,k} = \sum_{j=0}^{k-1} \beta^{k-j-1} g_{i,k}.
\label{m_s}
\end{equation}
It holds that $m_{i,k} \leq \frac{1}{1-\beta} g_{i,k}$, where we use $\sum_{j=0}^{k-1} \beta^{k-j-1} \leq \frac{1}{1-\beta}$. First, we show that the distance $ \Vert \bar{x}_{k+1}-\bar{x}_k \Vert$ is bounded by the consensus error and $\Vert \mathbf{v}_{k} \Vert$.

\begin{lemma}
\label{keyx}
    Let $x_{i,k+1}=\mathcal{P}_{\mathcal{M}}\left(\sum_{j=1}^n W_{i j}^t x_{j,k}-\alpha v_{i,k}\right)$, where $v_{i,k} \in T_{x_{i,k}} \mathcal{M}$. Suppose that Assumption~\ref{weight} holds. It holds that
    \begin{equation}
    \begin{aligned}
    & \Vert \bar{x}_{k+1} - \bar{x}_k \Vert \leq \frac{8Q+\sqrt{n} F + M}{n} \Vert \mathbf{x}_k - \bar{\mathbf{x}}_k \Vert^2 \\
    & + \frac{2 Q \alpha^2}{n}\Vert \mathbf{v}_{k} \Vert^2 +\alpha \Vert \hat{v}_{k} \Vert  + \frac{M}{n} \Vert \mathbf{x}_{k+1} - \bar{\mathbf{x}}_{k+1} \Vert^2.
    \end{aligned}
    \end{equation}
\end{lemma}
\begin{proof}
Since $\Vert \nabla \varphi^t(\mathbf{x}_k)\Vert=\Vert (\emph{I}_{nd}-\mathbf{W}^t)\mathbf{x}_k \Vert=\Vert (\emph{I}_{nd}-\mathbf{W}^t)(\mathbf{x}_k-\bar{\mathbf{x}}_k) \Vert \leq 2\Vert \mathbf{x}_k-\bar{\mathbf{x}}_k \Vert$, we have
\begin{equation}
\begin{aligned}
    & \left\|\hat{x}_{k+1}-\hat{x}_k\right\| \\
    & \leq\left\|\hat{x}_{k+1}-\hat{x}_k+\frac{1}{n} \sum_{i=1}^n\left(\operatorname{grad} \varphi_i^t\left(\mathbf{x}_k\right)+\alpha v_{i, k}\right)\right\| \\
    & +\left\|\frac{1}{n} \sum_{i=1}^n\left(\operatorname{grad} \varphi_i^t\left(\mathbf{x}_k\right)+\alpha v_{i, k}\right)\right\| \\
    & \leq \frac{Q}{n} \sum_{i=1}^n\left\|\nabla \varphi_i^t\left(\mathbf{x}_k\right)+\alpha v_{i, k}\right\|^2 \\
    & +\left\|\frac{1}{n} \sum_{i=1}^n \operatorname{grad} \varphi_i^t\left(\mathbf{x}_k\right)\right\|+\alpha \Vert \hat{v}_{k} \Vert \\
    & \leq \frac{2 Q}{n}\left\|\nabla \varphi^t(\mathbf{x}_k)\right\|^2+\frac{2 Q \alpha^2}{n}\left\|\mathbf{v}_{k}\right\|^2 \\
    & +\left\|\frac{1}{n} \sum_{i=1}^n \operatorname{grad} \varphi_i^t\left(\mathbf{x}_k\right)\right\|+\alpha \Vert \hat{v}_{k} \Vert \\
    & \leq \frac{8Q+\sqrt{n} F}{n} \Vert \mathbf{x}_k - \bar{\mathbf{x}}_k \Vert^2 + \frac{2 Q \alpha^2}{n}\Vert \mathbf{v}_{k} \Vert^2+\alpha \Vert \hat{v}_{k} \Vert ,
\end{aligned}
\label{y_x}
\end{equation}
where the second inequality uses the fact that $\Vert\mathcal{P}_{\mathcal{M}}(x+u)-x -\mathcal{P}_{T_x \mathcal{M}}(u) \Vert \leq Q\Vert u \Vert^2$ proposed by~\cite{deng2023decentralized} and the fourth inequality uses Lemma~\ref{f}. Under Lemma~\ref{m}, it holds that
\begin{equation}
\begin{aligned}
    & \Vert \bar{x}_{k+1} - \bar{x}_k \Vert \\
    & \leq \Vert \hat{x}_{k+1} - \hat{x}_k \Vert + \Vert \hat{x}_k - \bar{x}_k \Vert + \Vert \hat{x}_{k+1} - \bar{x}_{k+1} \Vert \\
    & \leq \frac{8Q+\sqrt{n} F + M}{n} \Vert \mathbf{x}_k - \bar{\mathbf{x}}_k \Vert^2 + \frac{2 Q \alpha^2}{n}\Vert \mathbf{v}_{k} \Vert^2 \\
    & +\alpha \Vert \hat{v}_{k} \Vert + \frac{M}{n} \Vert \mathbf{x}_{k+1} - \bar{\mathbf{x}}_{k+1} \Vert^2.
\end{aligned}
\end{equation}
The proof is completed.
\end{proof}

Next, we show the boundedness of consensus error.

\begin{lemma}
\label{key1}
    Under Assumptions~\ref{weight} and \ref{lipschitz}. Let $\{\mathbf{x}_k\}$ be generated by Algorithm~\ref{alg2}. For all $k$, $\mathbf{x}_k \in \mathcal{N}$, it holds that
    \begin{equation}
        \Vert \mathbf{x}_{k+1} - \bar{\mathbf{x}}_{k+1} \Vert \leq \sigma_2^t R \Vert \mathbf{x}_{k} - \bar{\mathbf{x}}_{k} \Vert + \alpha R \Vert \mathbf{v}_{k} \Vert .
    \end{equation}
\end{lemma}
\begin{proof}
Let $\mathcal{P}_{\mathcal{M}^n}(\mathbf{x})^\top=[\mathcal{P}_{\mathcal{M}}(x_1)^\top,\cdots,\mathcal{P}_{\mathcal{M}}(x_n)^\top]$. By the definition of $\bar{\mathbf{x}}_{k}$, we have
\begin{equation}
\begin{aligned}
    & \Vert \mathbf{x}_{k+1} - \bar{\mathbf{x}}_{k+1} \Vert \leq \Vert  \mathbf{x}_{k+1} - \bar{\mathbf{x}}_{k} \Vert \\
    & = \Vert  \mathcal{P}_{\mathcal{M}^n}\left(\mathbf{W}^t \mathbf{x}_{k}-\alpha \mathbf{v}_{k}\right) - \mathcal{P}_{\mathcal{M}^n}\left(\hat{\mathbf{x}}_{k}\right) \Vert \\
    & \leq R \left\| \mathbf{W}^t \mathbf{x}_{k}-\alpha \mathbf{v}_{k} - \hat{\mathbf{x}}_{k} \right\| \\
    & \leq \sigma_2^t R \Vert \mathbf{x}_{k} - \bar{\mathbf{x}}_{k} \Vert + \alpha R \Vert \mathbf{v}_{k} \Vert ,
\end{aligned}
\label{step2}
\end{equation}
where the first inequality follows from the optimality of $\bar{\mathbf{x}}_{k+1}$ and the second inequality utilizes Eq.(\ref{property}). The proof is completed.
\end{proof}

With the above lemma, we yield the relationship between the step-size and consensus error.

\begin{figure*}[ht]
	\centering
	\begin{minipage}{0.44\linewidth}
		\centering
		\includegraphics[width=1\linewidth]{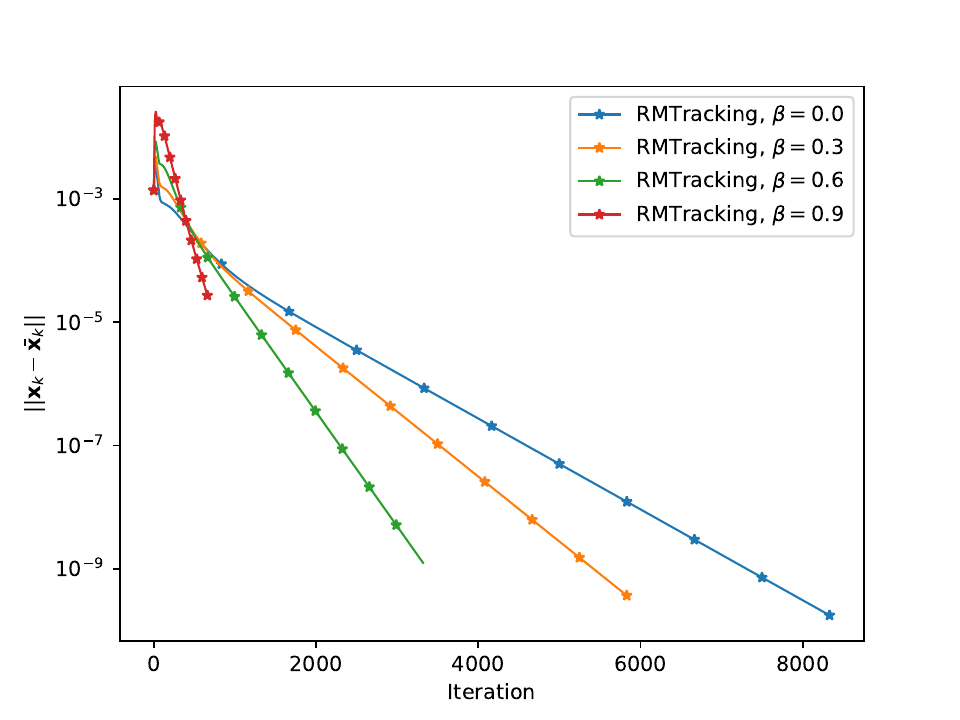}
	\end{minipage}
	\centering
	\begin{minipage}{0.44\linewidth}
		\centering
		\includegraphics[width=1\linewidth]{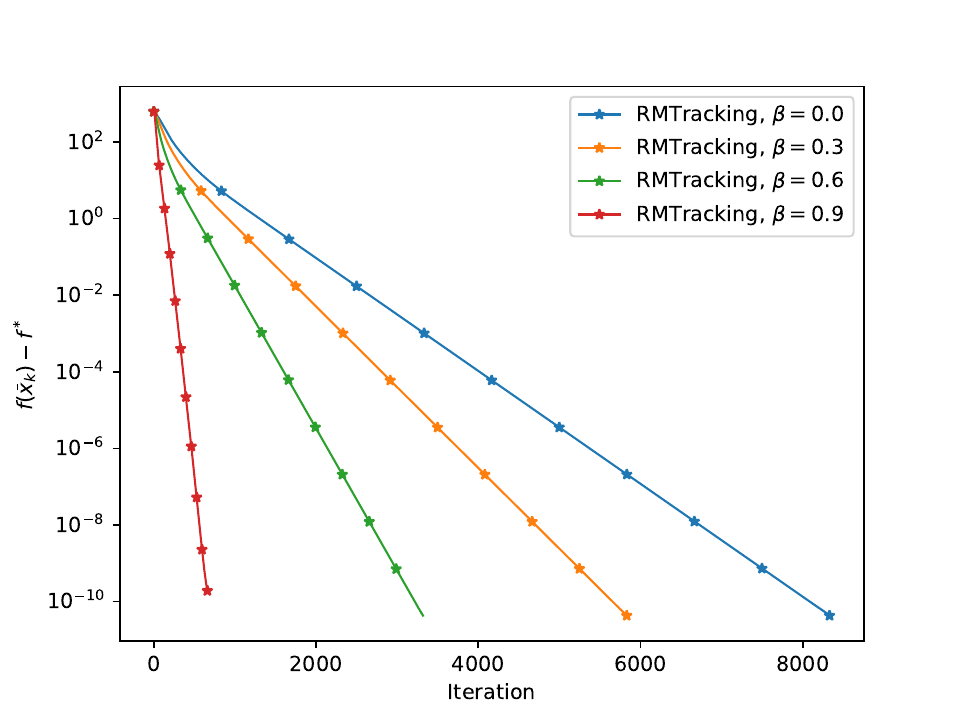}
	\end{minipage}
        \centering
	\begin{minipage}{0.44\linewidth}
		\centering
		\includegraphics[width=1\linewidth]{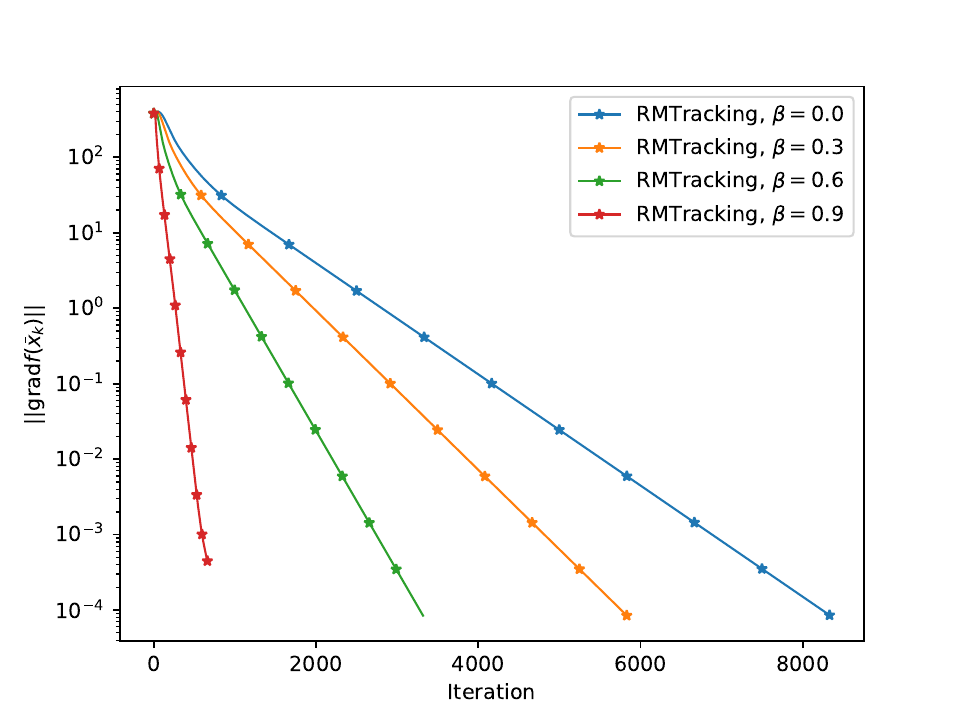}
	\end{minipage}
	\centering
	\begin{minipage}{0.44\linewidth}
		\centering
		\includegraphics[width=1\linewidth]{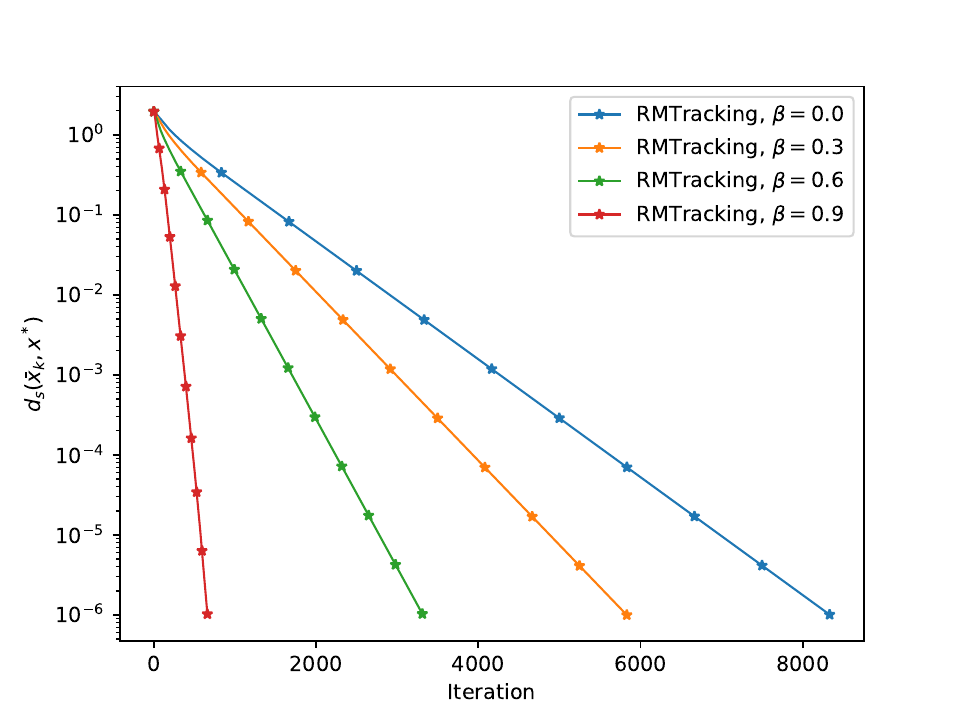}
	\end{minipage}
	\caption{Numerical results on synthetic data with different momentum weights and single-step consensus, eigengap $\Delta = 0.8$, Graph: Ring, $n=16$, $\hat{\alpha}=0.02$.}
	\label{fig1}
\end{figure*}

\begin{lemma}
Under Assumptions~\ref{weight} and \ref{lipschitz}. Let $\{\mathbf{x}_k\}$ be generated by Algorithm~\ref{alg2}. For all $k$, $\mathbf{x}_k \in \mathcal{N}$, we have

\begin{equation}
\frac{1}{n}\Vert \bar{\mathbf{x}}_k - \mathbf{x}_k\Vert^2 \leq C \frac{R^2}{(1-\beta)^2} L_g^2 \alpha^2.
\end{equation}
\label{lem5}
\end{lemma}
\begin{proof}
Based on Lemma~\ref{key1}, we yield
\begin{equation}
\Vert \mathbf{x}_{k+1} - \bar{\mathbf{x}}_{k+1} \Vert \leq \sigma_2^t R \Vert \mathbf{x}_{k} - \bar{\mathbf{x}}_{k} \Vert + \frac{4\sqrt{n}}{1-\beta} \alpha R L_g,
\label{error1}
\end{equation}
where it follows from Lemma~\ref{s} that $\Vert\mathbf{v}_k \Vert \leq \Vert\mathbf{s}_k \Vert \leq \frac{4 \sqrt{n}}{1-\beta} L_g$.

Let $\rho_t=\sigma_2^t R$ where $0<\rho_t <1$, it follows from Eq.(\ref{error1}) that
\begin{equation}
\begin{aligned}
&\left\|\mathbf{x}_{k+1}-\bar{\mathbf{x}}_{k+1}\right\| \\
& \leq \rho_t \left\|\mathbf{x}_k-\bar{\mathbf{x}}_k\right\|+\frac{4\sqrt{n}}{1-\beta} \alpha R L_g \\
& \leq \rho_t^{k+1} \left\|\mathbf{x}_0-\bar{\mathbf{x}}_0\right\|+\frac{4 \sqrt{n} \alpha R L_g}{1-\beta}  \sum_{l=0}^k \rho_t^{k-l} .
\end{aligned}
\label{error1}
\end{equation}

Let $y_k=\frac{\left\|\mathbf{x}_k-\bar{\mathbf{x}}_k\right\|}{\sqrt{n} \alpha}$. For a positive integer $K \leq k$, it follows from Eq.(\ref{error1}) that
\begin{equation}
\begin{aligned}
y_{k+1} &\leq \rho_t y_k + \frac{4 R L_g}{1-\beta}  \\
& \leq \rho_t^{k+1-K} y_K + \frac{4 R L_g}{1-\beta} \sum_{l=0}^k \rho_t^{k-l}.
\end{aligned}
\end{equation}
For $0\leq k \leq K$, there exists some $C'>0$ such that $y_k^2 \leq C' \frac{R^2}{(1-\beta)^2} L_g^2$, where $C'$ is independent of $L_g$ and $n$. For $k \geq K$, one has that $y_k^2 \leq C \frac{R^2}{(1-\beta)^2} L_g^2$, where $C=2C'+\frac{32}{(1-\rho_t)^2}$. Hence, we get $\frac{\left\|\mathbf{x}_k-\bar{\mathbf{x}}_k\right\|^2}{n} \leq C \frac{\alpha^2 R^2}{(1-\beta)^2} L_g^2$ for all $k\geq 0$, where $C=\mathcal{O}(\frac{1}{(1-\rho_t)^2})$. The proof is completed.
\end{proof}

\subsection{Convergence Analysis}

With above preparations, this part establishes the convergence rate analysis of the RMTracking algorithm. 

By leveraging Lipschitz-type inequalities for the momentum term on any compact $C^2$-submanifold in Euclidean space, we yield a decreasing boundedness on $f$. 

\begin{lemma}
\label{last}
Under Assumptions~\ref{weight} and \ref{lipschitz}. Let $\{\mathbf{x}_k\}$ be generated by Algorithm~\ref{alg2}. Then we have
\begin{equation}
\begin{aligned}
&f(\bar{x}_{k+1}) \leq f(\bar{x}_{k}) -\frac{\alpha}{2(1-\beta)} \Vert \hat{g}_k \Vert^2 + \mathcal{C}_1 \alpha^2 \frac{1}{n} \Vert \mathbf{s}_k \Vert^2 \\
& +  \mathcal{C}_2 \frac{1}{n} \Vert \mathbf{x}_k - \bar{\mathbf{x}}_k \Vert^2 + \mathcal{C}_3 \frac{1}{n} \Vert \mathbf{x}_{k+1} - \bar{\mathbf{x}}_{k+1} \Vert^2,
\end{aligned}
\end{equation}
where
\[
\begin{aligned}
 & \mathcal{C}_1 = 2L_g Q+6L_g + \frac{192 Q^2 L_g^3 \alpha^2}{(1-\beta)^2} \\
 & \mathcal{C}_2 = \frac{16 Q+19}{2} L_g + \frac{CM^2 R^2 L_g^2 \alpha}{1-\beta} + \frac{3C R^2 L_g^3 \alpha^2 (8Q+\sqrt{n} F + M)^2}{(1-\beta)^2} \\
 & \mathcal{C}_3 = \frac{CM^2 R^2 L_g^2 \alpha}{1-\beta} + \frac{3C M^2 R^2 L_g^3 \alpha^2}{(1-\beta)^2}.
\end{aligned}
\]
\end{lemma}
\begin{proof}
    The proofs can be found in Appendix~\ref{app2}.
\end{proof}

\begin{figure*}[ht]
	\centering
	\begin{minipage}{0.44\linewidth}
		\centering
		\includegraphics[width=1\linewidth]{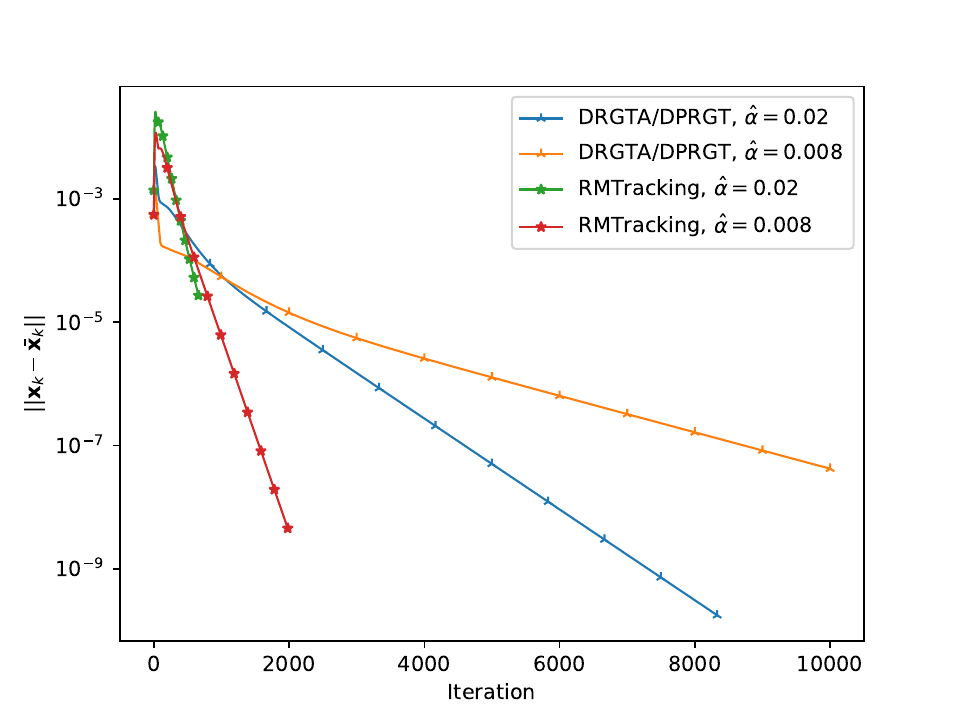}
	\end{minipage}
        \centering
	\begin{minipage}{0.44\linewidth}
		\centering
		\includegraphics[width=1\linewidth]{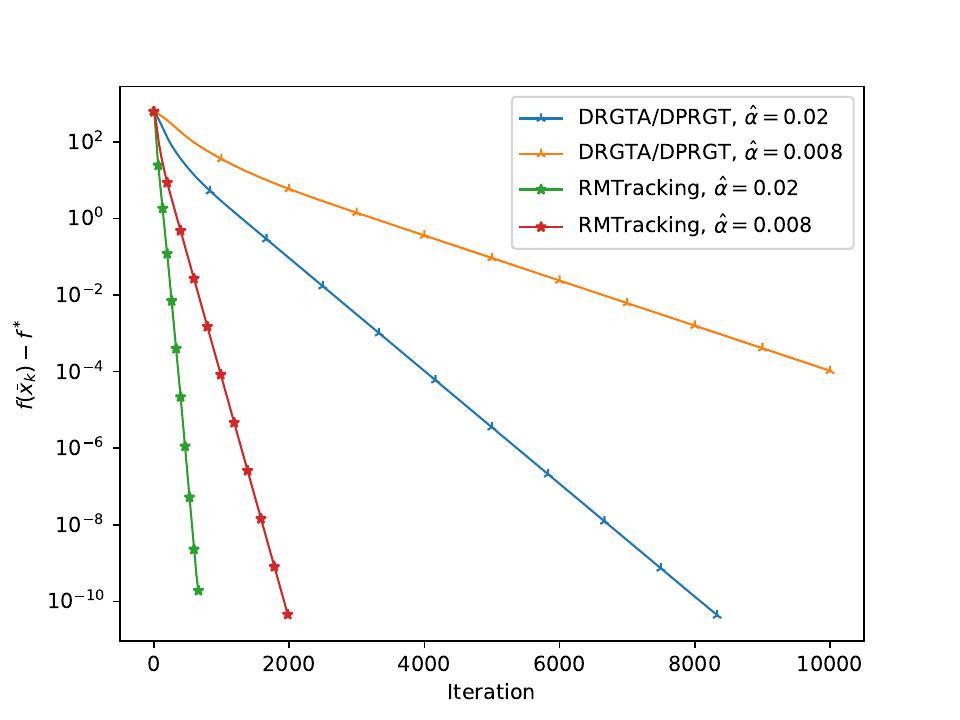}
	\end{minipage}
	\centering
	\begin{minipage}{0.44\linewidth}
		\centering
		\includegraphics[width=1\linewidth]{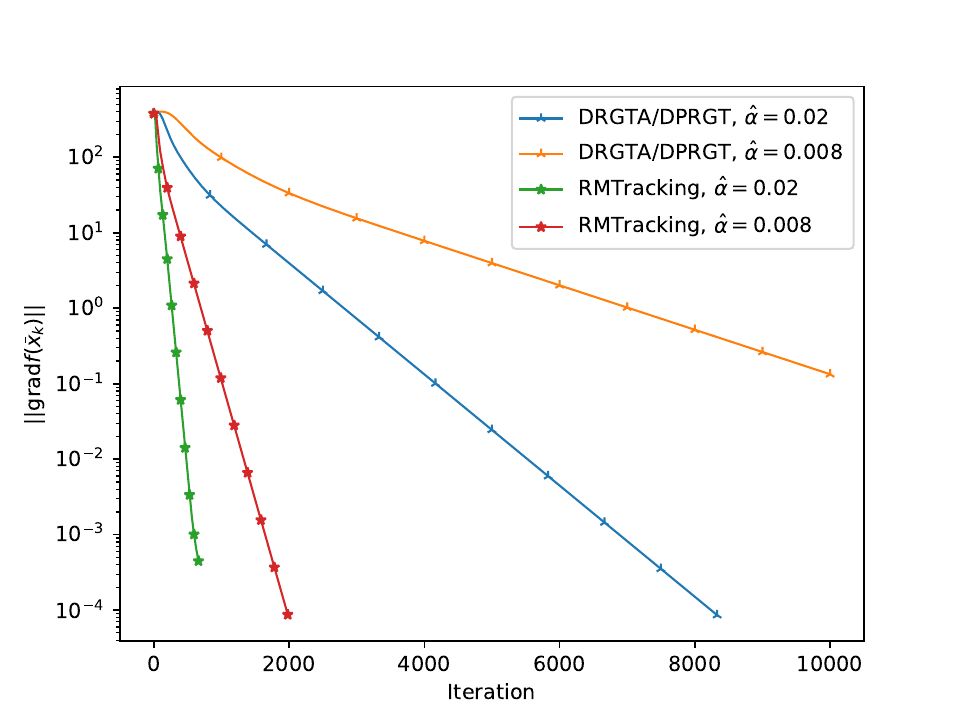}
	\end{minipage}
	\centering
	\begin{minipage}{0.44\linewidth}
		\centering
		\includegraphics[width=1\linewidth]{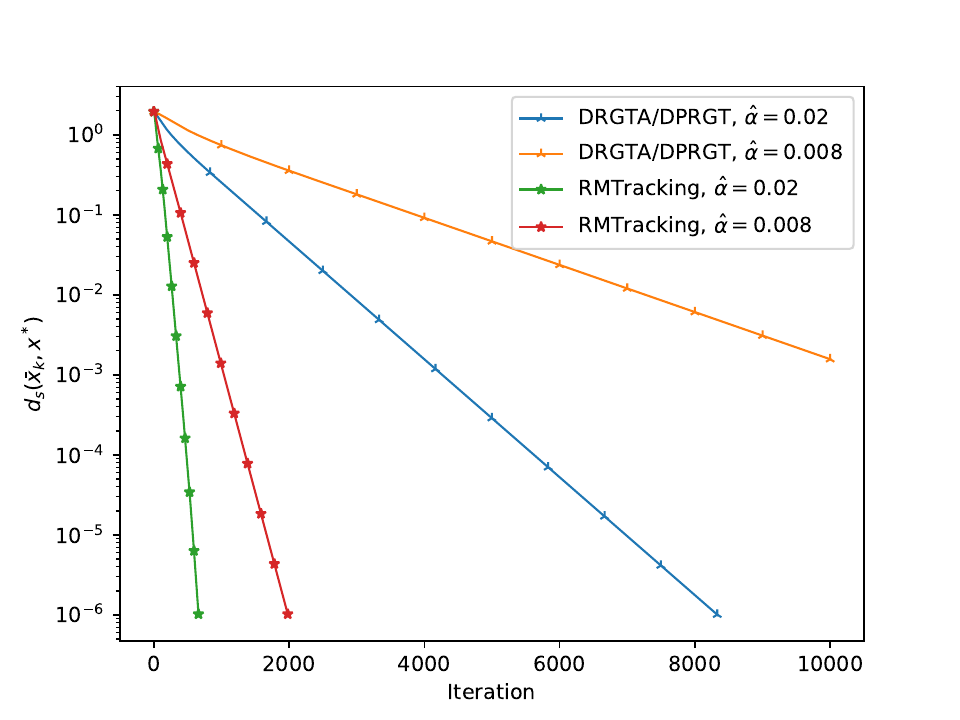}
	\end{minipage}
	\caption{Numerical results on synthetic data with different step-sizes and single-step consensus, eigengap $\Delta = 0.8$, Graph: Ring, $n=16$.}
	\label{fig2}
\end{figure*}

Under the safety step-size, we present the following main theorem on $\mathcal{O}(\frac{1-\beta}{K})$ convergence rate of the Riemannian gradient average to reach the $\epsilon$-stationary point of (\ref{decentralized}).

\begin{theorem}
\label{final}
    Under Assumptions~\ref{weight} and \ref{lipschitz}. Let $\{\mathbf{x}_k\}$ be generated by Algorithm~\ref{alg2}. If $\mathbf{x}_0 \in \mathcal{N}$, $t\geq \max\left\{\log_{\sigma^2}\left(\frac{1}{4\sqrt{n}}\right)\right\}$, and $\alpha$ satisfies
    \begin{equation}
       \alpha < \min \left\{1,\frac{(1-\beta)^2}{32\mathcal{C}_1 (1-\beta) + 128 L_g^2 \tilde{\mathcal{C}}_3(1+4R^2)} \right\} , 
    \end{equation}
    then we have
    \begin{equation}
    \begin{aligned}
        & \min_{k=0,\cdots,K} \frac{1}{n} \Vert \hat{g}_k \Vert^2 = \mathcal{O}\left(\frac{1-\beta}{\alpha K}\right), \\
        & \min_{k=0,\cdots,K} \frac{1}{n} \Vert \mathbf{x}_{k} - \bar{\mathbf{x}}_{k} \Vert^2 = \mathcal{O}\left(\frac{1}{K}\right), \\
        &\min_{k=0,\cdots,K} \Vert \operatorname{grad} f(\bar{x}_k) \Vert^2 = \mathcal{O}\left(\frac{1}{\alpha K}\right).
    \end{aligned}
    \end{equation}
    When $\frac{\alpha}{1-\beta}$ is sufficiently small, it holds that
    \begin{equation}
    \begin{aligned}
        & \min_{k=0,\cdots,K} \frac{1}{n} \Vert \hat{g}_k \Vert^2 = \mathcal{O}\left(\frac{1-\beta}{\alpha K}\right), \\
        & \min_{k=0,\cdots,K} \frac{1}{n} \Vert \mathbf{x}_{k} - \bar{\mathbf{x}}_{k} \Vert^2 = \mathcal{O}\left(\frac{1-\beta}{K}\right), \\
        &\min_{k=0,\cdots,K} \Vert \operatorname{grad} f(\bar{x}_k) \Vert^2 = \mathcal{O}\left(\frac{1-\beta}{\alpha K}\right).
    \end{aligned}
    \end{equation}
\end{theorem}
\begin{proof}
    The proofs can be found in Appendix~\ref{app3}.
\end{proof}

In particular, RMTracking exhibits an $\mathcal{O}(\frac{1-\beta}{K})$ convergence rate when the step-size is sufficiently small, which is $\frac{1}{1-\beta}$ times faster than other Riemannian gradient tracking algorithms (DRGTA and DPRGT). When $\beta$ is set to $0$, RMTracking will degenerate into Riemannian Gradient Tracking, which produces results consistent with DRGTA and DPRGT.

\section{Numerical Experiments}

In this section, we compare our RMTracking method with DRGTA~\cite{chen2021decentralized} and DPRGT~\cite{deng2023decentralized}, both of which are first-order distributed Riemannian optimization approaches utilizing retraction and projection. We evaluate their performance on the distributed eigenvector problem given by:
\begin{equation}
\begin{aligned}
& \min _{\mathbf{x} \in \mathcal{M}^n}-\frac{1}{2 n} \sum_{i=1}^n \operatorname{tr}\left(x_i^{\top} A_i^{\top} A_i x_i\right), \quad \text { s.t. } \; x_1=x_2=\ldots=x_n ,\\
& \text{where} \quad \mathcal{M}^n:=\underbrace{\operatorname{St}(d, r) \times \operatorname{St}(d, r) \times \cdots \times \operatorname{St}(d, r)}_n .
\label{pca}
\end{aligned}
\end{equation}
Here, $A_i \in \mathbb{R}^{m_i \times d}$ represents the local data matrix, and $m_i$ denotes the sample size for agent $i$. Building upon this, $A^\top:=[A_1^\top, \cdots, A_n^\top]$ denotes the global data matrix. For any solution $x^*$ of Eq.(\ref{pca}), $x^* O$ also qualifies as a solution in essence for an arbitrary orthogonal matrix $O \in \mathbb{R}^{r \times r}$. Consequently, the distance between two points $x$ and $x^*$ can be defined as
\begin{equation}
d_s(x,x^*)=\min_{O^\top O=OO^\top=\emph{I}_r} \Vert x O - x^*\Vert.
\end{equation}

We measure these algorithms using four metrics: the consensus error ``$ \Vert \mathbf{x}_k - \mathbf{\bar{x}}_k \Vert$", the gradient norm ``$\Vert \mathrm{grad} f( {\bar{x}}_k) \Vert$", the difference of the objective function ``$f( {\bar{x}}_k) - f^*$", and the distance to the global optimum ``$d_s({\bar{x}}_k, x^*)$". The experiments are conducted using the Intel(R) Core(TM) i7-14700K CPU, and the code is implemented in Python under the package of mpi4py.

\subsection{Synthetic data}

We run $n=16$ and set $m_1=m_2=\cdots=m_{16}=1000$, $d=10$, and $r=5$. Subsequently, we generate $m_1 \times n$ i.i.d samples to obtain $A$ by following standard multi-variate Gaussian distribution. Specifically, let $A= \mathbb{U} \Sigma \mathbb{V}^\top$ represent the truncated SVD, where $\mathbb{V} \in \mathbb{R}^{d \times d}$ and $\mathbb{U} \in \mathbb{R}^{1000n \times d}$ are orthogonal matrices, and $\Sigma \in \mathbb{R}^{d \times d}$ is a diagonal matrix. Finally, we define the singular values of $A$ as $\Sigma_{i,i}=\Sigma_{0,0} \times \Delta^{i/2}$, where $i \in [d]$ and the eigengap satisfies $\Delta \in (0,1)$. 

We use a constant step-size for all comparisons. For DRGTA, DPRGT, and RMTracking, we cap the maximum epoch at 10,000, terminating early if $d_s(\bar{x}_k,x^*) \leq 10^{-6}$. The step-size is defined as $\alpha=\frac{n \hat{\alpha}}{\sum_{i=1}^n m_i}$. We evaluate the graph matrix representing the topology among agents (Ring network). According to \cite{chen2021decentralized}, $W$ is the Metroplis constant matrix~\cite{shi2015extra}. 

We conduct ablation experiments on RMTracking under different momentum weights ($\beta=0.0,0.3,0.6,0.9$). Note that RMTracking is equivalent to Riemannian Gradient Tracking used in DRGTA and DPRGT when $\beta=0.0$. Let the graph among 8 agents be Ring network. Furthermore, we set the single-step consensus $t=1$ and step-size $\hat{\alpha}=0.02$, respectively. Due to the introduction of momentum, the empirical results in Figure \ref{fig1} show that our algorithm converges faster as $\beta \in [0,1)$ becomes larger, which is consistent with the convergence rate of $\mathcal{O}(\frac{1-\beta}{K})$. We further compare our algorithm with the state-of-the-art distributed Riemannian gradient tracking algorithms DRGTA and DPRGT, as shown in Figure \ref{fig2}. Here, we apply the large step-sizes ($\hat{\alpha}=0.008,0.02$) and $\beta=0.9$. To achieve an $\epsilon$-stationary point with the same distance, RMTracking requires much less iterations than the other two algorithms.

\section{Conclusion}

This paper proposes a Riemannian Momentum Tracking (RMTracking) algorithm over a compact submanifold. If the agents are restricted to a suitable neighborhood, we prove an $\mathcal{O}(\frac{1-\beta}{K})$ convergence rate of the Riemannian gradient average using fixed step-sizes. Empirically, our algorithm achieves $\frac{1}{1-\beta}$ times faster than other related algorithms when the step-size is sufficiently small. Finally, we conduct numerical experiments to demonstrate the effectiveness of the algorithm and the theoretical analysis.
The current work's improvement in convergence rate remains constrained by parameters $\alpha$ and $\beta$. In the future, we plan to further investigate the acceleration mechanisms of distributed Riemannian optimization to achieve better convergence rates.

\appendix
\subsection{Two key inequalities}
\label{ine}
In the following two lemmas, we show that $\Vert \bar{x} - \hat{x} \Vert$ and $\left\| \sum_{i=1}^n \operatorname{grad} \varphi^t_i(\mathbf{x}) \right\|$ are bounded by the square of consensus error, which are useful in proving Lemma \ref{keyx} and \ref{last}.

\begin{lemma}
\label{m}
    \cite{deng2023decentralized} Let $M= \max_{x \in \mathcal{M}} \Vert D^2 \mathcal{P}_{\mathcal{M}}(x) \Vert$. For any $\mathbf{x} \in \mathcal{M}^n$, we have
    \begin{equation}
        \Vert \bar{x} - \hat{x} \Vert \leq M \frac{\Vert \mathbf{x} - \bar{\mathbf{x}}\Vert^2}{n}.
    \end{equation}
\end{lemma}

\begin{lemma}
\label{f}
    \cite{deng2023decentralized} Let $F= 2\max_{x \in \mathcal{M}} \Vert D^2 \mathcal{P}_{T_x \mathcal{M}}(\cdot) \Vert$. For any $\mathbf{x} \in \mathcal{M}^n$, we have
    \begin{equation}
    \left\| \sum_{i=1}^n \operatorname{grad} \varphi^t_i(\mathbf{x}) \right\| \leq \sqrt{n} F \Vert \mathbf{x}-\bar{\mathbf{x}} \Vert^2 .
    \end{equation}
\end{lemma}

\subsection{Uniform boundedness}

In the following two lemmas, we investigate the uniform boundedness of $\Vert s_{i,k} \Vert$ and $\Vert m_{i,k} \Vert$, which is useful in proving Lemma \ref{last}.

\begin{lemma}
\label{mom}
    Under Assumptions~\ref{weight} and \ref{lipschitz}. Let $\{\mathbf{x}_k\}$ be generated by Algorithm~\ref{alg2}. For all $k$, $\mathbf{x}_k \in \mathcal{N}$, we have \\
    (i) Uniform boundedness of $\Vert m_{i,k} \Vert$:
    \begin{equation}
    \Vert m_{i,k}\Vert \leq \frac{L_g}{1-\beta}, \quad \forall i \in[n],
    \end{equation}
    (ii) Summed boundedness of the square of $\Vert m_{i,k} \Vert$:
    \begin{equation}
    \sum_{k=0}^K \Vert m_{i,k} \Vert^2 \leq \frac{2}{(1-\beta)^2} \sum_{k=0}^K \Vert g_{i,k} \Vert^2.
    \end{equation}
    (iii) Summed boundedness of the square of $\Vert \mathbf{m}_{k+1}-\mathbf{m}_{k} \Vert$:
    \begin{equation}
    \sum_{k=0}^K \Vert \mathbf{m}_{k+1}-\mathbf{m}_{k} \Vert^2 \leq \frac{2L_g^2}{(1-\beta)^2} \sum_{k=0}^K \Vert \mathbf{x}_{k+1}-\mathbf{x}_{k} \Vert^2.
    \end{equation}
\end{lemma}

\begin{proof}
Based on Lemma~\ref{lem1}, we have $\Vert \operatorname{grad} f_i\left(x_{i, k}\right) \Vert  \leq L_g$. It follows from Eq.(\ref{m_s}) that
\begin{equation}
\begin{aligned}
    & \Vert m_{i,k}\Vert \\
    & \leq (1 + \beta + \beta^2 + \cdots + \beta^{k-1}) \max_{j=0, \cdots,k-1} \Vert g_{i,j} \Vert \\
    & \leq \frac{1}{1-\beta} \max_{j=0, \cdots,k-1} \Vert g_{i,j} \leq \Vert \leq \frac{L_g}{1-\beta} .
\end{aligned}
\end{equation}
According to Eq.(\ref{momentum}), we have
\begin{equation}
     \Vert m_{i,k+1} \Vert \leq \beta \Vert m_{i,k} \Vert + \Vert g_{i,k} \Vert.
\end{equation}
Since $\beta \in [0,1)$, applying the key lemma in~\cite{xu2015augmented} yields
\begin{equation}
    \sum_{k=0}^K \Vert m_{i,k} \Vert^2 \leq \frac{2}{(1-\beta)^2} \sum_{k=0}^K \Vert g_{i,k} \Vert^2 + \frac{2}{1-\beta^2} \Vert m_{i,0} \Vert^2 ,
\end{equation}
where $m_{i,0}=0$. Furthermore, it follows from Lemma \ref{lem1} that
\begin{equation}
\begin{aligned}
     & \Vert \mathbf{m}_{k+1}-\mathbf{m}_{k} \Vert \leq \beta \Vert \mathbf{m}_{k}-\mathbf{m}_{k-1} \Vert + \Vert \mathbf{g}_{k}-\mathbf{g}_{k-1} \Vert \\
     & \leq \beta \Vert \mathbf{m}_{k}-\mathbf{m}_{k-1} \Vert + L_g \Vert \mathbf{x}_{k}-\mathbf{x}_{k-1} \Vert .
\end{aligned}
\end{equation}
Similarly, we have
\begin{equation}
\begin{aligned}
    & \sum_{k=0}^K \Vert \mathbf{m}_{k+1}-\mathbf{m}_{k} \Vert^2 \\
    & \leq \frac{2L_g^2}{(1-\beta)^2} \sum_{k=0}^K \Vert \mathbf{x}_{k+1}-\mathbf{x}_{k} \Vert^2 + \frac{2}{1-\beta^2} \Vert \mathbf{m}_{1}-\mathbf{m}_{0} \Vert^2 .
\end{aligned}
\end{equation}
The proof is completed.
\end{proof}

\begin{lemma}
\label{s}
    Under Assumptions~\ref{weight} and \ref{lipschitz}. Let $\{\mathbf{x}_k\}$ be generated by Algorithm~\ref{alg2}. If $\mathbf{x}_0 \in \mathcal{N}$ and $t\geq \max\left\{\log_{\sigma^2}\left(\frac{1}{4\sqrt{n}}\right)\right\}$, it follows that
    \begin{equation}
        \Vert s_{i,k}\Vert \leq \frac{4L_g}{1-\beta}, \quad \forall i \in[n].
    \end{equation}
\end{lemma}

\begin{proof}
We prove it by induction on both $\Vert s_{i,k} \Vert$ and $\Vert \hat{x}_k - \bar{x}_k \Vert$. Then we have $\Vert s_{i,0}\Vert=\Vert \operatorname{grad} f_i\left(x_{i, 0}\right) \Vert \leq L_g \leq \frac{L_g}{1-\beta}$ for all $i \in [n]$ and $\Vert \hat{x}_0 - \bar{x}_0 \Vert \leq \frac{1}{2}\Gamma$. Suppose for some $k \geq 0$ that $\Vert s_{i,k}\Vert \leq \frac{4L_g}{1-\beta}$ and $\Vert \hat{x}_k - \bar{x}_k \Vert \leq \frac{1}{2}\Gamma$. We have
\begin{equation}
\begin{aligned}
&\left\| s_{i, k+1}-\hat{g}_k \right\| \\
&=\left\| \sum_{j=1}^n W_{i j}^t s_{j,k} - \hat{g}_k + m_{i, k+1}-m_{i, k}\right\| \\
& \leq \left\|\sum_{j=1}^n \left(W_{i j}^t-\frac{1}{n}\right) s_{j,k} \right\|+\left\|m_{i, k+1}-m_{i, k}\right\| \\
& \leq \sigma_2^t \sqrt{n} \max _i\left\|s_{i, k}\right\| + \Vert m_{i, k+1}\Vert + \Vert m_{i, k} \Vert \\
& \leq \sigma_2^t \sqrt{n} \max _i\left\|s_{i, k}\right\|+\frac{2L_g}{1-\beta} \\
& \leq \frac{1}{4} \max _i\left\|s_{i, k}\right\|+ \frac{2L_g}{1-\beta} \leq \frac{L_g}{1-\beta}+\frac{2L_g}{1-\beta} \leq \frac{3L_g}{1-\beta} .
\end{aligned}
\end{equation}
Hence, $\Vert s_{i, k+1}\Vert \leq \left\|s_{i, k+1}-\hat{g}_k\right\| + \left\|\hat{g}_k\right\| \leq \frac{4L_g}{1-\beta}$, where we use $\Vert \hat{g}_k \Vert \leq \frac{1}{n} \sum_{i=1}^n \Vert \operatorname{grad} f_i\left(x_{i, k}\right)\Vert \leq L_g \leq \frac{L_g}{1-\beta}$. The proof is completed.
\end{proof}

With these lemmas, we give the following lemma on the summed boundedness of $ \Vert \hat{g}_k \Vert^2$, which is a key step to prove the convergence rate in Theorem \ref{final}.

\begin{lemma}
\label{g_g}
    Under Assumptions~\ref{weight} and \ref{lipschitz}. Let $\{\mathbf{x}_k\}$ be generated by Algorithm~\ref{alg2}. If $\mathbf{x}_0 \in \mathcal{N}$, it follows that
    \begin{equation}
    \begin{aligned}
        & -\sum_{k=0}^K \Vert \hat{g}_k \Vert^2 \leq 32L_g^2\tilde{\mathcal{C}}_3 \frac{1}{n} \sum_{k=0}^K \Vert \mathbf{x}_k-\bar{\mathbf{x}}_k \Vert^2 + \frac{\tilde{\mathcal{C}}_4 (1-\beta)^2}{2} \\
        & + \left(8 L_g^2 \alpha^2 \tilde{\mathcal{C}}_3 - \frac{(1-\beta)^2}{4}\right) \frac{1}{n} \sum_{k=0}^K \Vert \mathbf{s}_k \Vert^2 .
    \end{aligned}
    \end{equation}
    
\end{lemma}

\begin{proof}
Since $\Vert \mathbf{s}_k \Vert \leq \Vert \mathbf{s}_k - \hat{\mathbf{m}}_k \Vert + \Vert \hat{\mathbf{m}}_k \Vert$, it follows from Lemma \ref{mom} that

\begin{equation}
\begin{aligned}
    & -\sum_{k=0}^K \Vert \hat{g}_k \Vert^2 \\
    & \leq -\frac{(1-\beta)^2}{2} \sum_{k=0}^K \Vert \hat{m}_k \Vert^2 = -\frac{(1-\beta)^2}{2n} \sum_{k=0}^K \Vert \hat{\mathbf{m}}_k \Vert^2 \\
    & \leq \frac{(1-\beta)^2}{2n} \sum_{k=0}^K \Vert \mathbf{s}_k - \hat{\mathbf{m}}_k \Vert^2 - \frac{(1-\beta)^2}{4n} \sum_{k=0}^K \Vert \mathbf{s}_k \Vert^2 .
\end{aligned}
\end{equation}
On the other hand, it follows from the definition of $\hat{\mathbf{m}}_{k+1}$ that
\begin{equation}
    \Vert \mathbf{s}_{k+1} - \hat{\mathbf{m}}_{k+1} \Vert \leq \sigma_2^t \Vert \mathbf{s}_k - \hat{\mathbf{m}}_k \Vert + \Vert \mathbf{m}_{k+1} - \mathbf{m}_k \Vert .
\label{s_g}
\end{equation}
Since $\sigma_2^t \in (0,1)$, applying the key lemma in~\cite{xu2015augmented} to Eq.(\ref{s_g}) yields
\begin{equation}
    \frac{1}{n} \sum_{k=0}^K \Vert \mathbf{s}_k - \hat{\mathbf{m}}_k \Vert^2 \leq \tilde{\mathcal{C}}_3 \frac{1}{n} \sum_{k=0}^K \Vert \mathbf{m}_{k+1} - \mathbf{m}_k \Vert^2 + \tilde{\mathcal{C}}_4 ,
\label{s_g2}
\end{equation}
where
\begin{equation}
\tilde{\mathcal{C}}_3 = \frac{2}{(1-\sigma_2^t)^2} ,\;\; \tilde{\mathcal{C}}_4 = \frac{2}{1 - (\sigma_2^t)^2} \frac{1}{n} \Vert \mathbf{s}_0 - \hat{\mathbf{m}}_0 \Vert^2 .
\end{equation}
Plugging Lemma \ref{mom} into Eq.(\ref{s_g2}) gives
\begin{equation}
\begin{aligned}
    & \frac{1}{n} \sum_{k=0}^K \Vert \mathbf{s}_k - \hat{\mathbf{m}}_k \Vert^2 \leq \frac{2\tilde{\mathcal{C}}_3 L_g^2}{(1-\beta)^2} \frac{1}{n} \sum_{k=0}^K \Vert \mathbf{x}_{k+1} - \mathbf{x}_k \Vert^2 + \tilde{\mathcal{C}}_4 \\
    & \leq \frac{64\tilde{\mathcal{C}}_3 L_g^2}{(1-\beta)^2} \frac{1}{n} \sum_{k=0}^K \Vert \mathbf{x}_k-\bar{\mathbf{x}}_k \Vert^2 + \frac{16\alpha^2\tilde{\mathcal{C}}_3 L_g^2}{(1-\beta)^2} \frac{1}{n} \sum_{k=0}^K \Vert \mathbf{s}_k \Vert^2 + \tilde{\mathcal{C}}_4,
\end{aligned}
\end{equation}
where the second inequality uses
\begin{equation}
\begin{aligned}
    \Vert \mathbf{x}_{k+1} - \mathbf{x}_k \Vert &\leq \Vert \mathbf{x}_{k+1} - \mathbf{x}_k + (\emph{I}_{nd}-\mathbf{W}^t)\mathbf{x}_k +\alpha \mathbf{v}_{k} \Vert \\
    & + \Vert (\emph{I}_{nd}-\mathbf{W}^t)\mathbf{x}_k +\alpha \mathbf{v}_{k} \Vert \\
    & \leq 2 \Vert (\emph{I}_{nd}-\mathbf{W}^t)\mathbf{x}_k +\alpha \mathbf{v}_{k} \Vert \\
    & \leq 4 \Vert \mathbf{x}_k-\bar{\mathbf{x}}_k \Vert + 2 \alpha \Vert \mathbf{v}_{k} \Vert .
\end{aligned}
\end{equation}
The proof is completed. 
\end{proof}

\subsection{Proofs for Lemma~\ref{last}}
\label{app2}
\begin{proof}
It follows from Lemma~\ref{lem1} that
\begin{equation}
\begin{aligned}
& \Vert \hat{g}_k - \operatorname{grad} f(\bar{x}_{k}) \Vert^2 \leq \frac{1}{n} \sum^n_{i=1} \Vert \operatorname{grad} f_i(x_{k}) - \operatorname{grad} f_i(\bar{x}_{k}) \Vert^2 \\
& \leq \frac{L_g^2}{n} \Vert \mathbf{x}_k - \bar{\mathbf{x}}_k \Vert^2 .
\end{aligned}
\end{equation}
By Young's inequality, we get
\begin{equation}
\begin{aligned}
    & \left\langle\operatorname{grad} f(\bar{x}_{k}) - \hat{g}_k, \bar{x}_{k+1} - \bar{x}_k \right\rangle \\
    & \leq \frac{L_g}{n} \Vert \mathbf{x}_k - \bar{\mathbf{x}}_k \Vert^2 + \frac{L_g}{4} \Vert \bar{x}_{k+1} - \bar{x}_k \Vert^2 .
\end{aligned}
\end{equation}
Therefore, we have
\begin{equation}
\begin{aligned}
    & f(\bar{x}_{k+1}) \\
    & \leq f(\bar{x}_{k}) + \left\langle\operatorname{grad} f(\bar{x}_{k}), \bar{x}_{k+1} - \bar{x}_k \right\rangle + \frac{L_g}{2} \Vert \bar{x}_{k+1} - \bar{x}_k \Vert^2 \\
    & = f(\bar{x}_{k}) + \left\langle\operatorname{grad} f(\bar{x}_{k}) - \hat{g}_k, \bar{x}_{k+1} - \bar{x}_k \right\rangle \\
    & + \left\langle \hat{g}_k, \bar{x}_{k+1} - \bar{x}_k \right\rangle + \frac{L_g}{2} \Vert \bar{x}_{k+1} - \bar{x}_k \Vert^2 \\
    & \leq f(\bar{x}_{k}) + \frac{L_g}{n} \Vert \mathbf{x}_k - \bar{\mathbf{x}}_k \Vert^2 + \frac{3L_g}{4} \Vert \bar{x}_{k+1} - \bar{x}_k \Vert^2 \\
    & + \left\langle \hat{g}_k, \bar{x}_{k+1} - \bar{x}_k \right\rangle \\
    & = f(\bar{x}_{k}) + \frac{L_g}{n} \Vert \mathbf{x}_k - \bar{\mathbf{x}}_k \Vert^2 + \frac{3L_g}{4} \Vert \bar{x}_{k+1} - \bar{x}_k \Vert^2 \\
    & + \left\langle \hat{g}_k, \hat{x}_{k+1} - \hat{x}_k \right\rangle  + \left\langle \hat{g}_k, \bar{x}_{k+1} - \bar{x}_k + \hat{x}_k - \hat{x}_{k+1} \right\rangle .
\end{aligned}
\end{equation}
By Young's inequality, we obtain
\begin{equation}
\begin{aligned}
    & \left\langle \hat{g}_k, \bar{x}_{k+1} - \bar{x}_k + \hat{x}_k - \hat{x}_{k+1} \right\rangle \\
    & \leq \frac{\alpha}{2(1-\beta)} \Vert \hat{g}_k \Vert^2 + \frac{1-\beta}{\alpha} \left(\Vert \hat{x}_{k} - \bar{x}_k \Vert^2 + \Vert \hat{x}_{k+1}- \bar{x}_{k+1} \Vert^2 \right) \\
    & \leq \frac{\alpha}{2(1-\beta)} \Vert \hat{g}_k \Vert^2 + \frac{M^2(1-\beta)}{n^2 \alpha } \left(\Vert \mathbf{x}_k - \bar{\mathbf{x}}_k \Vert^4 + \Vert \mathbf{x}_{k+1} - \bar{\mathbf{x}}_{k+1} \Vert^4 \right),
\end{aligned}
\end{equation}
where the second inequality utilizes Lemma~\ref{m}. Since $\frac{1}{n}\Vert \bar{\mathbf{x}}_k - \mathbf{x}_k\Vert^2 \leq C \frac{R^2}{(1-\beta)^2} L_g^2 \alpha^2$, we have
\begin{equation}
\begin{aligned}
    & f(\bar{x}_{k+1}) \\
    & \leq f(\bar{x}_{k})  + \left\langle \hat{g}_k, \hat{y}_{k+1} - \hat{x}_k \right\rangle  + \frac{3L_g}{4} \Vert \bar{x}_{k+1} - \bar{x}_k \Vert^2 \\
    &  + \left(\frac{L_g}{n} + \frac{CM^2 R^2 L_g^2 \alpha}{n(1-\beta)} \right) \Vert \mathbf{x}_k - \bar{\mathbf{x}}_k \Vert^2 \\
    & + \frac{\alpha}{2(1-\beta)} \Vert \hat{g}_k \Vert^2 + \frac{CM^2 R^2 L_g^2 \alpha}{n(1-\beta)}\Vert \mathbf{x}_{k+1} - \bar{\mathbf{x}}_{k+1} \Vert^2 .
\end{aligned}
\label{f1}
\end{equation}
For $\Vert \bar{x}_{k+1} - \bar{x}_k \Vert^2$, it follows from Lemma~\ref{keyx} that
\begin{equation}
\begin{aligned}
    & \Vert \bar{x}_{k+1} - \bar{x}_k \Vert^2 \\
    & \leq \frac{4(8Q+\sqrt{n} F + M)^2}{n^2} \Vert \mathbf{x}_k - \bar{\mathbf{x}}_k \Vert^4 + 4\alpha^2 \Vert \hat{v}_{k} \Vert^2  \\
    & + \frac{4 M^2}{n^2} \Vert \mathbf{x}_{k+1} - \bar{\mathbf{x}}_{k+1} \Vert^4 + \frac{16 Q^2 \alpha^4}{n^2}\Vert \mathbf{v}_{k} \Vert^4 \\
    & \leq \frac{4C R^2 L_g^2 (8Q+\sqrt{n} F + M)^2}{n (1-\beta)^2} \alpha^2 \Vert \mathbf{x}_k - \bar{\mathbf{x}}_k \Vert^2 + \frac{4\alpha^2}{n} \Vert \mathbf{v}_{k} \Vert^2 \\
    & + \frac{4C M^2 R^2 L_g^2 \alpha^2}{n (1-\beta)^2}  \Vert \mathbf{x}_{k+1} - \bar{\mathbf{x}}_{k+1} \Vert^2 + \frac{16 Q^2 \alpha^4}{n^2}\Vert \mathbf{v}_{k} \Vert^4 \\
    & \leq \frac{4C R^2 L_g^2 (8Q+\sqrt{n} F + M)^2}{n (1-\beta)^2} \alpha^2 \Vert \mathbf{x}_k - \bar{\mathbf{x}}_k \Vert^2 \\
    & + \frac{4C M^2 R^2 L_g^2 \alpha^2}{n (1-\beta)^2}  \Vert \mathbf{x}_{k+1} - \bar{\mathbf{x}}_{k+1} \Vert^2  \\
    & + \left(\frac{256 Q^2 L_g^2 \alpha^4}{n(1-\beta)^2} + \frac{4\alpha^2}{n} \right)\Vert \mathbf{v}_{k} \Vert^2 .
\end{aligned}
\label{f2}
\end{equation}
For $\left\langle \hat{g}_k, \hat{x}_{k+1} - \hat{x}_k \right\rangle$, it follows from Eq. (\ref{m_s}) that
\begin{equation}
\begin{aligned}
    & \left\langle \hat{g}_k, \hat{x}_{k+1} - \hat{x}_k \right\rangle \\
    & =\left\langle\hat{g}_k, \frac{1}{n} \sum_{i=1}^n\left(x_{i, k+1}-x_{i, k}+\alpha s_{i, k}+\nabla \varphi_i^t\left(\mathbf{x}_k\right)\right)\right\rangle \\
    & -\left\langle\hat{g}_k, \frac{1}{n} \sum_{i=1}^n\left(\alpha s_{i, k}+\nabla \varphi_i^t\left(\mathbf{x}_k\right)\right)\right\rangle \\
    & \leq\left\langle\hat{g}_k, \frac{1}{n} \sum_{i=1}^n\left(x_{i, k+1}-x_{i, k}+\alpha v_{i, k}+\nabla \varphi_i^t\left(\mathbf{x}_k\right)\right)\right\rangle \\
    & +\left\langle\hat{g}_k, \frac{1}{n} \sum_{i=1}^n \alpha\left(s_{i, k}-v_{i, k}\right)\right\rangle-\frac{\alpha}{1-\beta}\left\|\hat{g}_k\right\|^2 .
\end{aligned}
\label{f4}
\end{equation}
Since $s_{i,k} - v_{i,k} \in N_{x_{i,k}} \mathcal{M}$, we have
\begin{equation}
\begin{aligned}
    & \left\langle\hat{g}_k, \frac{1}{n} \sum_{i=1}^n \alpha\left(s_{i, k}-v_{i, k}\right)\right\rangle \\
    & =\frac{\alpha}{n} \sum_{i=1}^n\left\langle\hat{g}_k-\operatorname{grad} f_i\left(x_{i, k}\right), s_{i, k}-v_{i, k}\right\rangle \\
    & \leq  \frac{1}{4 n L_g} \sum_{i=1}^n\left\|\hat{g}_k-\operatorname{grad} f_i\left(x_{i, k}\right)\right\|^2 \\
    & +\frac{\alpha^2 L_g}{n} \sum_{i=1}^n\left\|\mathcal{P}_{N_{x_{i, k}} \mathcal{M}}\left(s_{i, k}\right)\right\|^2 \\
    & \leq  \frac{1}{4 n^2 L_g} \sum_{i=1}^n \sum_{j=1}^n\left\|\operatorname{grad} f_j\left(x_{j, k}\right)-\operatorname{grad} f_i\left(x_{i, k}\right)\right\|^2 \\
    & +\frac{\alpha^2 L_g}{n}\left\|\mathbf{s}_k\right\|^2 \\
    & \leq \frac{L_g}{4n}\left\|\mathbf{x}_k-\bar{\mathbf{x}}_k\right\|^2+\frac{\alpha^2 L_g}{n}\left\|\mathbf{s}_k\right\|^2.
\end{aligned}
\label{f5}
\end{equation}
Let $\nabla \varphi^t_i(\mathbf{x}_k)+\alpha v_{i,k}=d_{i,1}+d_{i,2}$, where $d_{i,1}=\mathcal{P}_{T_{x_{i,k} }\mathcal{M}}(\nabla \varphi^t_i(\mathbf{x}_k)+\alpha v_{i,k})$ and $d_{i,2}=\nabla \varphi^t_i(\mathbf{x}_k)+\alpha v_{i,k}-d_{i,1}$. Therefore, 
\begin{equation}
\begin{aligned}
    & \left\langle\hat{g}_k, \frac{1}{n} \sum_{i=1}^n\left(x_{i, k+1}-x_{i, k}+\alpha v_{i, k}+\nabla \varphi_i^t\left(\mathbf{x}_k\right)\right)\right\rangle \\
    & = \frac{1}{n} \sum_{i=1}^n\left\langle\hat{g}_k, \mathcal{P}_{\mathcal{M}}\left(x_{i, k}-d_{i, 1}-d_{i, 2}\right)-\left[x_{i, k}-d_{i, 1}\right]\right\rangle \\
    & +\frac{1}{n} \sum_{i=1}^n\left\langle\hat{g}_k, d_{i, 2}\right\rangle \\
    & \leq \frac{L_g Q}{n} \sum_{i=1}^n\left\|d_i\right\|^2+\frac{1}{n} \sum_{i=1}^n\left\langle\hat{g}_k-\operatorname{grad} f_i\left(x_{i, k}\right), d_{i, 2}\right\rangle \\
    & \leq \frac{L_g Q}{n} \sum_{i=1}^n\left\|d_i\right\|^2+\frac{1}{4 n L_g} \sum_{i=1}^n\left\|\hat{g}_k-\operatorname{grad} f_i\left(x_{i, k}\right)\right\|^2 \\
    & +\frac{L_g}{n} \sum_{i=1}^n\left\|d_{i, 2}\right\|^2 \\
    & \leq \frac{L_g Q+L_g}{n}\left\|\left(I_{n d}-\mathbf{W}^t\right) \mathbf{x}_k+\alpha \mathbf{v}_{k}\right\|^2+\frac{L_g}{4n}\left\|\bar{\mathbf{x}}_k-\mathbf{x}_k\right\|^2 \\
    & \leq \frac{32 L_g Q+33 L_g}{4n}\left\|\bar{\mathbf{x}}_k-\mathbf{x}_k\right\|^2+\frac{2 L_g Q+2 L_g}{n} \alpha^2\left\|\mathbf{s}_{k}\right\|^2 .
\end{aligned}
\label{f6}
\end{equation}
Plugging Eq.(\ref{f5}) and Eq.(\ref{f6}) into Eq.(\ref{f4}) gives
\begin{equation}
\begin{aligned}
    & \left\langle \hat{g}_k, \hat{y}_{k+1} - \hat{x}_k \right\rangle \leq \frac{32L_g Q+34L_g}{4n}\left\|\bar{\mathbf{x}}_k-\mathbf{x}_k\right\|^2 \\
    & +\frac{2L_g Q+3L_g}{n} \alpha^2\left\|\mathbf{s}_{k}\right\|^2 -  \frac{\alpha}{1-\beta} \Vert \hat{g}_k \Vert^2 .
\label{f7}
\end{aligned}
\end{equation}
Then, combining Eq.(\ref{f2}) and Eq.(\ref{f7}) gives
\begin{equation}
\begin{aligned}
    & f(\bar{x}_{k+1}) \leq f(\bar{x}_{k}) + \left(\frac{L_g}{n} + \frac{CM^2 R^2 L_g^2 \alpha}{n(1-\beta)} \right) \Vert \mathbf{x}_k - \bar{\mathbf{x}}_k \Vert^2  \\
    & + \frac{CM^2 R^2 L_g^2 \alpha}{n(1-\beta)}\Vert \mathbf{x}_{k+1} - \bar{\mathbf{x}}_{k+1} \Vert^2 - \frac{\alpha}{2(1-\beta)} \Vert \hat{g}_k \Vert^2 \\
    & + \frac{32L_g Q+34L_g}{4n} \left\|\bar{\mathbf{x}}_k-\mathbf{x}_k\right\|^2+\frac{2L_g Q+3L_g}{n}\alpha^2\left\|\mathbf{s}_{k}\right\|^2  \\
    & + \left(\frac{192 Q^2 L_g^3 \alpha^4}{n(1-\beta)^2} + \frac{3 L_g\alpha^2}{n} \right)\Vert \mathbf{s}_{k} \Vert^2 \\
    & + \frac{3C M^2 R^2 L_g^3 \alpha^2}{n(1-\beta)^2}  \Vert \mathbf{x}_{k+1} - \bar{\mathbf{x}}_{k+1} \Vert^2 \\
    & + \frac{3C R^2 L_g^3 (8Q+\sqrt{n} F + M)^2}{n(1-\beta)^2} \alpha^2 \Vert \mathbf{x}_k - \bar{\mathbf{x}}_k \Vert^2.
\end{aligned}
\end{equation}
The proof is completed.
\end{proof}

\subsection{Proofs for Theorem~\ref{final}}
\label{app3}
\begin{proof}
It follows from Lemma~\ref{key1} that
\begin{equation}
    \Vert \mathbf{x}_{k+1} - \bar{\mathbf{x}}_{k+1} \Vert \leq \sigma_2^t R \Vert \mathbf{x}_{k} - \bar{\mathbf{x}}_{k} \Vert + \alpha R \Vert \mathbf{s}_{k} \Vert .
\end{equation}
Since $\sigma_2^t R \in (0,1)$, applying the key lemma in~\cite{xu2015augmented} gives
\begin{equation}
    \frac{1}{n} \sum_{k=0}^{K+1} \Vert \mathbf{x}_{k} - \bar{\mathbf{x}}_{k} \Vert^2 \leq \tilde{\mathcal{C}}_1 \alpha^2 R^2 \frac{1}{n} \sum_{k=0}^{K+1} \Vert \mathbf{s}_{k} \Vert^2 + \tilde{\mathcal{C}}_2 ,
\label{x_m}
\end{equation}
where
\begin{equation}
    \tilde{\mathcal{C}}_1=\frac{2}{\left(1-\sigma_2^t R\right)^2}, \tilde{\mathcal{C}}_2=\frac{2}{1-\left(\sigma_2^t R\right)^2}\frac{1}{n} \Vert \mathbf{x}_{0} - \bar{\mathbf{x}}_{0} \Vert^2 .
\end{equation}
It follows from Lemma~\ref{last} that
\begin{equation}
\begin{aligned}
& f(\bar{x}_{K+1}) \leq f(\bar{x}_{0}) -\frac{\alpha}{2(1-\beta)} \sum_{k=0}^K \Vert \hat{g}_k \Vert^2 + \mathcal{C}_1 \alpha^2 \frac{1}{n} \sum_{k=0}^K \Vert \mathbf{s}_{k} \Vert^2 \\
& + (\mathcal{C}_2+\mathcal{C}_3) \frac{1}{n} \sum_{k=0}^{K+1} \Vert \mathbf{x}_k - \bar{\mathbf{x}}_k \Vert^2 \\
& \leq f(\bar{x}_{0}) + \mathcal{C}_1 \alpha^2 \frac{1}{n} \sum_{k=0}^K \Vert \mathbf{s}_k \Vert^2 + (\mathcal{C}_2+\mathcal{C}_3) \frac{1}{n} \sum_{k=0}^{K+1} \Vert \mathbf{x}_k - \bar{\mathbf{x}}_k \Vert^2 \\
& + \frac{\tilde{\mathcal{C}}_4 \alpha (1-\beta)}{4} + \frac{16L_g^2\alpha\tilde{\mathcal{C}}_3}{1-\beta} \frac{1}{n} \sum_{k=0}^K  \Vert \mathbf{x}_k-\bar{\mathbf{x}}_k \Vert^2 \\
& + \left(\frac{4 L_g^2 \alpha^3 \tilde{\mathcal{C}}_3}{1-\beta} - \frac{(1-\beta) \alpha}{8} \right) \frac{1}{n} \sum_{k=0}^K \Vert \mathbf{s}_k \Vert^2, 
\end{aligned}
\end{equation}
where the second inequality uses Lemma~\ref{g_g}. It follows from Eq.(\ref{x_m}) that
\begin{equation}
\begin{aligned}
& f(\bar{x}_{K+1}) \leq f(\bar{x}_{0}) + \mathcal{C}_1 \alpha^2 \frac{1}{n} \sum_{k=0}^K \Vert \mathbf{s}_k \Vert^2 + \frac{\tilde{\mathcal{C}}_4(1-\beta) \alpha}{4} \\
& + \left(\frac{4 L_g^2 \alpha^3 \tilde{\mathcal{C}}_3}{1-\beta} - \frac{(1-\beta) \alpha}{8} \right) \frac{1}{n} \sum_{k=0}^K \Vert \mathbf{s}_k \Vert^2 \\
& + \tilde{\mathcal{C}}_1 \alpha^2 R^2\left(\frac{16L_g^2\alpha\tilde{\mathcal{C}}_3}{1-\beta} + \mathcal{C}_2+\mathcal{C}_3\right) \frac{1}{n} \sum_{k=0}^K \Vert \mathbf{s}_k \Vert^2 \\
& + \tilde{\mathcal{C}}_2\left(\frac{16L_g^2\alpha\tilde{\mathcal{C}}_3}{1-\beta} + \mathcal{C}_2+\mathcal{C}_3\right) + \frac{16\tilde{\mathcal{C}}_1(\mathcal{C}_2+\mathcal{C}_3) \alpha^2 R^2 L_g^2}{(1-\beta)^2} .
\end{aligned}
\end{equation}
Let $\mathcal{C}_4=\tilde{\mathcal{C}}_2\left(\frac{16L_g^2\alpha\tilde{\mathcal{C}}_3}{1-\beta} + \mathcal{C}_2+\mathcal{C}_3\right) + \frac{16\tilde{\mathcal{C}}_1(\mathcal{C}_2+\mathcal{C}_3) \alpha^2 R^2 L_g^2}{(1-\beta)^2}$. Since $\alpha < \min \left\{1,\frac{(1-\beta)^2}{32\mathcal{C}_1 (1-\beta) + 128 L_g^2 \tilde{\mathcal{C}}_3(1+4R^2)} \right\}$, it follows from Lemma~\ref{mom} that
\begin{equation}
\begin{aligned}
    & f(\bar{x}_{K+1}) \leq f(\bar{x}_{0}) + \mathcal{C}_4 + \frac{\tilde{\mathcal{C}}_4 (1-\beta) \alpha}{4} \\
    & - \frac{(1-\beta) \alpha}{4} \left(\frac{1}{2} - \frac{8\mathcal{C}_1 \alpha}{1-\beta} - \frac{32L_g^2\alpha^2\tilde{\mathcal{C}}_3(1+4R^2)}{(1-\beta)^2} \right) \frac{1}{n} \sum_{k=0}^K \Vert \mathbf{s}_k \Vert^2 \\
    & \leq f(\bar{x}_{0}) - \frac{\alpha(1-\beta)}{16} \frac{1}{n} \sum_{k=0}^K \Vert \mathbf{s}_k \Vert^2 + \mathcal{C}_4 + \frac{\tilde{\mathcal{C}}_4 (1-\beta) \alpha}{4}.
\end{aligned}
\end{equation}
This implies
\begin{equation}
    \min_{k=0,\cdots,K} \frac{1}{n} \Vert \mathbf{s}_{k} \Vert^2 \leq \frac{16(f(\bar{x}_{0}) - f^* + \mathcal{C}_5)}{\alpha (1-\beta) K} ,
\end{equation}
where $\mathcal{C}_5=\mathcal{C}_4 + \frac{\tilde{\mathcal{C}}_4 (1-\beta) \alpha}{4}$ and $f^*=\min_{x \in \mathcal{M}} f(x)$. Using Lemma \ref{g_g}, it holds that
\begin{equation}
\begin{aligned}
    & \min_{k=0,\cdots,K} \Vert \hat{g}_k \Vert^2 \leq \frac{(1-\beta)^2}{2}\min_{k=0,\cdots,K} \frac{1}{n} \Vert \mathbf{s}_{k} \Vert^2 \\
    & \leq \frac{8(1-\beta)(f(\bar{x}_{0}) - f^* + \mathcal{C}_5)}{\alpha K},
\end{aligned}
\end{equation}
Then, it follows from Eq.(\ref{x_m}) that
\begin{equation}
\begin{aligned}
    & \min_{k=0,\cdots,K} \frac{1}{n}\Vert \mathbf{x}_{k} - \bar{\mathbf{x}}_{k} \Vert^2 \\
    & \leq \frac{(1-\beta) \left(\frac{16 \tilde{\mathcal{C}}_1 R^2 \alpha (f(\bar{x}_{0}) - f^* + \mathcal{C}_5)}{1-\beta} + \tilde{\mathcal{C}}_2\right)}{K} .
\end{aligned}
\end{equation}
When $\frac{\alpha}{1-\beta}$ is sufficiently small, $\mathcal{O}\left(\frac{(1-\beta) \tilde{\mathcal{C}}_2}{K}\right)$ will dominate the convergence rate. Since $\Vert \operatorname{grad} f(\bar{x}_k) \Vert^2 \leq 2 \Vert \hat{g}_k \Vert^2 + 2\Vert \operatorname{grad} f(\bar{x}_k) - \hat{g}_k \Vert^2 \leq 2 \Vert \hat{g}_k \Vert^2 + \frac{2L_g^2}{n} \Vert \mathbf{x}_{k} - \bar{\mathbf{x}}_{k} \Vert^2$, we have
\begin{equation}
\begin{aligned}
& \min_{k=0,\cdots,K} \Vert \operatorname{grad} f(\bar{x}_k) \Vert^2 \\
& \leq \frac{(1-\beta)}{\alpha K}\left[\left(16+\frac{32\tilde{\mathcal{C}}_1 R^2 L_g^2 \alpha^2}{1-\beta}\right) (f(\bar{x}_{0}) - f^*) + \right. \\
& \left(16+\frac{32\tilde{\mathcal{C}}_1 R^2 L_g^2 \alpha^2}{1-\beta}\right) \left(\hat{\mathcal{C}}_1+\hat{\mathcal{C}}_2(1-\beta)\alpha + \hat{\mathcal{C}}_3\frac{\alpha}{1-\beta}+ \right. \\
& \left.\left. \hat{\mathcal{C}}_4\frac{\alpha^2}{(1-\beta)^2} + \hat{\mathcal{C}}_5\frac{\alpha^3}{(1-\beta)^3} + \hat{\mathcal{C}}_6\frac{\alpha^4}{(1-\beta)^4} \right) + 2L_g^2\tilde{\mathcal{C}}_2 \alpha \right],
\end{aligned}
\end{equation}
where
\[
\begin{aligned}
& \hat{\mathcal{C}}_1 = (16 Q+19) \tilde{\mathcal{C}}_2 L_g / 2 \\
& \hat{\mathcal{C}}_2 = \tilde{\mathcal{C}}_4 /4 \\
& \hat{\mathcal{C}}_3 = 16L_g^2\tilde{\mathcal{C}}_2\tilde{\mathcal{C}}_3+2CM^2 R^2 L_g^2\tilde{\mathcal{C}}_2 \\
& \hat{\mathcal{C}}_4 = 3C R^2 L_g^3\tilde{\mathcal{C}}_2 [(8Q+\sqrt{n} F + M)^2+M^2]+ 8\tilde{\mathcal{C}}_1 R^2 L_g^3 (16 Q+19) \\
& \hat{\mathcal{C}}_5 = 32 \tilde{\mathcal{C}}_1 R^4 L_g^4 C M^2 \\
& \hat{\mathcal{C}}_6 = 48C R^4 L_g^5\tilde{\mathcal{C}}_1 \tilde{\mathcal{C}}_2 [(8Q+\sqrt{n} F + M)^2+M^2]
\end{aligned}
\]
Note that $\tilde{\mathcal{C}}_1,\tilde{\mathcal{C}}_2,\tilde{\mathcal{C}}_3,\tilde{\mathcal{C}}_4$ are constants independent of $\alpha$ and $\beta$, and thus remains so for $\hat{\mathcal{C}}_1,\hat{\mathcal{C}}_2,\hat{\mathcal{C}}_3,\hat{\mathcal{C}}_4,\hat{\mathcal{C}}_5,\hat{\mathcal{C}}_6$ as well.
When $\frac{\alpha}{1-\beta}$ is sufficiently small, e.g., $\frac{\alpha}{1-\beta} \ll \min \left\{1,\frac{1}{2\tilde{\mathcal{C}}_1 R^2 L_g^2}\right\}$, $\mathcal{O}\left(\frac{(1-\beta)\left[16(f(\bar{x}_{0}) - f^* + \mathcal{C}_5) + 2L_g^2\tilde{\mathcal{C}}_2 \alpha \right]}{\alpha K}\right)$ will dominate the convergence rate. The proof is completed.
\end{proof}

\bibliographystyle{IEEEtran}
\bibliography{sample}

\end{document}